\documentclass[12pt, notitlepage]{article}

\usepackage[english]{babel}
\usepackage{amsmath}
\usepackage{amssymb}
\usepackage{amsfonts}
\usepackage{amsthm}
\usepackage{array}
\usepackage{calrsfs}
\usepackage{enumitem}
\usepackage{fancyhdr}
\usepackage[margin=1.25in]{geometry}
\usepackage{graphicx}
\usepackage[hidelinks]{hyperref}
\usepackage{inputenc}
\usepackage{etoolbox}
\usepackage{mdframed}
\usepackage{pgf}         
\usepackage{setspace}
\usepackage{tcolorbox}
\usepackage{thmtools}
\usepackage{tikz}
\usepackage{xcolor}
\usepackage{xparse}
\usepackage{xspace}
\usepackage{float}
\usetikzlibrary{automata, positioning, arrows}

\usepackage[maxbibnames = 10]{biblatex}
\usepackage{csquotes}
\usepackage{fontspec}

\usepackage[bigdelims]{newtxmath} 
\usepackage[cal=boondoxo]{mathalfa} 

\renewcommand{\mathbb}{\varmathbb}

\theoremstyle{plain}
\newtheorem{thm}{Theorem}[section]

\theoremstyle{definition}
\newtheorem{defin}{Definition}

\theoremstyle{plain}
\newtheorem{lemma}[thm]{Lemma}

\theoremstyle{remark}
\newtheorem*{rmk}{Remark}

\theoremstyle{plain}

\theoremstyle{plain}
\newtheorem{corollary}[thm]{Corollary}

\newcommand\vertiii[1]{{\left\vert\kern-0.25ex\left\vert\kern-0.25ex\left\vert #1 
    \right\vert\kern-0.25ex\right\vert\kern-0.25ex\right\vert}}

\newcommand\vertii[1]{{\left\vert\kern-0.25ex\left\vert #1 
    \right\vert\kern-0.25ex\right\vert}}

\NewDocumentCommand\esmes{O{E} O{\pazocal{A}} O{\mu}}{$(#1,#2,#3)$\xspace}

\NewDocumentCommand\esmesle{O{E} O{\pazocal{A}}}{$(#1,#2)$\xspace}

\NewDocumentCommand\enspol{O{A} O{X}}{#1[#2]\xspace}

\NewDocumentCommand\enspolk{O{K} O{X}} {\pazocal{#1}[#2]\xspace}

\NewDocumentCommand\enssf{O{A} O{X}}{#1[\kern-0.2ex[#2]\kern-0.2ex]\xspace}

\NewDocumentCommand{\cyl}{O{(t_0,x_0)} O{a} O{R} O{}}{\mathcal{C}#4(#1,#2,#3)}

\NewDocumentCommand\Lloc{m O{\Omega}}{\L^{#1}_{\text{loc}}(#2)}

\newcommand{\ind}[1]{1\kern-1ex 1_{#1}}

\NewDocumentCommand{\produit}{O{k} O{0} O{n}}{\prod_{#1 = #2}^{#3} \,}

\newcommand{\ps}[2]{\left( #1 ,\, #2 \right)}

\NewDocumentCommand{\psl}{m m O{\Omega}}{\ps{#1}{#2}_{\L^2(#3)}}

\NewDocumentCommand{\somme}{O{k} O{0} O{n}}{\sum_{#1 = #2}^{#3} \,}

\renewcommand{\L}{\pazocal L}

\newcommand{\R}{\mathbb R}

\newcommand{\conv}{\text{Conv}}

\newcommand{\vect}{\text{Span}}

\DeclareMathAlphabet{\pazocal}{OMS}{zplm}{m}{n}

\title{Equiprojective polytopes in higher dimension}

\author{Alice Cousaert \thanks{Université Sorbonne Paris Nord, CNRS, Laboratoire d'Informatique de Paris Nord, LIPN, F-93430 Villetaneuse, France ; cousaert@lipn.univ-paris13.fr} }

\date{}

\begin{document}

\maketitle

\begin{abstract}
    A 3-dimensional polytope is called $k$-equiprojective if every planar projection along a direction non-parallel to any facet is a $k$-gon. In this article, we generalise equiprojectivity to higher dimensions and give a lower bound on the number of combinatorial types of equiprojective polytopes. We also establish the pathwise connectedness of a subset of the Grassmannian in the case of $(d-2)$-dimensional spaces with conditions on the explicit path. This makes it possible to extend the Hasan--Lubiw characterisation of equiprojectivity to higher dimensions. Equiprojectivity provides cases relevant to the study of the Shadow Vertex algorithm, showing there is no hope minimising the complexity of the projection. It also offers a reverse point of view on the usual study of planar projections of polytopes as the projections have a fixed size.
\end{abstract}

\section{Introduction}

Equiprojective polytopes were considered by Shephard in his $9^{th}$ problem in~\cite{SHEPH} as 3\nobreakdash-dimensional polytopes such that almost every planar projections have the same combinatorial size (see \cite{MH} for examples of equiprojective polyhedra). This problem, restated by Hallard Croft, Kenneth Falconer and Richard Guy in \cite{Pb_in_geom}, concerns the construction of equiprojective 3\nobreakdash-dimensional polytopes. To answer this question, Masud Hasan and Anna Lubiw have established in~\cite{MHAL} a characterisation of equiprojectivity in terms of \emph{edge-facet}: a couple $(e,F)$ of an edge $e$ and a facet $F$ containing it, endowed with an orientation of $e$ with respect to $F$. Two distinct edge-facets $(e,F)$ and $(e',F')$ are said to be \emph{compensating} if $F$ and $F'$ are either the same facet or parallel facets and $e$ and $e'$ are parallel edges that have opposite orientations. They show that a polytope $P$ is equiprojective if and only if its edge-facets can be partitioned into compensating pairs. More recently, Théophile Buffière and Lionel Pournin established in \cite{TBLP} another characterisation of equiprojectivity using normal cones, allowing them to study the behaviour of equiprojectivity with respect to Minkowski sums and to deduce a lower bound on the number of combinatorial types of $k$-equiprojective polytopes for all integers $k$. These characterisations both rely on the fact that, as the direction of projection continuously moves, for equiprojective polytopes, the disappearance of an edge caused by the degeneracy of a facet is automatically compensated by the emergence of another parallel edge (located either on the same facet, or on an other parallel one).

The “almost every” projections we want to consider are those in a direction non-parallel to any facet, as they would otherwise send a facet to a segment, thus changing the combinatorial complexity of the projection. Our first endeavour will be to determine which are the planar projections we want to avoid in the high-dimensional case. This will lead us to a generalisation  equiprojectivity. In particular, we will see that high-dimensional zonotopes are still $k$-equiprojective with a known $k$ (this was remarked in dimension 3 in \cite{TBLP}), and if we fix the dimension $d$ that there are at least \begin{equation}
    \label{borne}
    k^{\frac k2\left(d^2-2d + O_{k\to + \infty}\left( \frac{1}{\log k} \right) \right)}
\end{equation} combinatorial types of $k$-equiprojective polytopes when $k$ is even.

The Hasan--Lubiw characterisation relies on a lemma stating that we can continuously modify the direction of projection to any other such that at any given time there are no more than one facet (and eventually its parallel counterpart) that is degenerating along the path. This means that it is enough to study the changes of the combinatorial size of the planar projections when exactly one pair of parallel facets is degenerating. In section 4 we develop a notion of pathwise conectedness with effective paths in the Grassmanian and we use it to establish the connectedness of a subset of the Grassmannian in the following theorem of independent interest:

\begin{thm}
    \label{informal}
    Given a $d$-dimensional polytope $P$, we can connect any two of its admissible orthogonal projections within the space of planar orthogonal projections while making sure that there are only finitely many degeneracies, and no two non-parallel 2-dimensional faces of $P$ are degenerating simultaneously.
\end{thm}
\noindent Where the notions of admissibility and degeneration are yet to be defined in section 2. For short, we will refer to $d$-dimensional polytope as $d$-polytope, and $k$-dimensional faces of a polytope $P$ as $k$-faces of $P$.

Finally, section 5 presents a generalisation of the edge-facet, then called \emph{edge-2\nobreakdash-faces} and its subtleties: it is the given of $(e,F,F_0)$ an edge $e$, a 2-face $F$ containing it and another 2-face $F_0$ parallel to $F$ such that we can see both of them on the boundary of a projection. In higher dimensions, we will see that we do not need to consider every possible pair of parallel 2-faces. The construction of the orientation and the notion of compensation will be done in the same fashion as what is done in~\cite{MHAL}. Two edge-2-faces $(e,F,F_0)$ and $(e',F',F'_0)$ are said to be compensating each other if $e$ and $e'$ are parallel and have opposite orientations and if either $F=F'$ and $F_0=F'_0$ or $F=F'_0$ and $F_0= F'$. We also present some technical lemmas relying on the aforementioned explicit paths, use them in order to precisely handle $d$\nobreakdash-polytopes the same way we would handle a 3-polytope, as well as the proof of the generalised Hasan--Lubiw characterisation:

\begin{thm}
    \label{charact}
    A $d$-polytope is equiprojective if and only if its edge-2-faces can be partitioned into compensating pairs.
\end{thm}

In some sense, equiprojectivity is the reverse point of view of several well-studied problems, such as the extension problem (see \cite{Exp_Low} and \cite{TR}), where we would like to simplify the study of a given polytope by seeing it as the projection of a higher-dimensional polytope. Another example of such reverse problem would be the work of Alexander Black and Francisco Criado in~\cite{Black}, where they fix the polytope and then study the expected size of a random projection, while in this article we fix the size of the planar projections and study properties on the polytope. Moreover, understanding equiprojectivity may help in the study of the Shadow Vertex algorithm (see \cite{DDSH} and also \cite{DDNH}), especially showing that in some cases, it is impossible to find a planar projection minimizing the combinatorial complexity of the shadow. A related problem asks for the combinatorics of the sections of a polytope by affine spaces \cite{affine_1, affine_2, affine_3, affine_4}. All of these problems deal with the interplay between convexity and combinatorics that is still elusive despite its importance in diverse areas, such as linear optimisation \cite{lin_op_3, lin_op_2, lin_op_1, lin_op_4, lien_1, lien_2}, geometry \cite{moser, SHEPH, SHEPH2, lien_3}, algebra \cite{affine_1, alg} or convex analysis \cite{conv_an, affine_4}.

\section{A definition of equiprojectivity in higher dimension}

The aim of this section is to generalise the definition of equiprojectivity, and especially to determine which are the planar projections we need to avoid. In this article, we will always assume that a polytope $P\subset \R^d$ is full-dimensional, up to working in the smallest affine space containing $P$. Our new definition would be:

\begin{defin}
    \label{new_def}
    A polytope $P\subset \R^d$ is said to be \emph{$k$-equiprojective} if and only if every projection on an “admissible plane” is a $k$-gon.
\end{defin}

The only task left is to discuss what “admissible plane” would mean in higher dimensions. In dimension 3, the reason why we want to exclude the projections parallel to a facet of $P$ is to be sure that no triplet of vertices are sent to aligned points on the boundary of the projection. If it is the case, then it means that a facet of $P$ (and its parallel counterpart if it exists) is sent on an edge of the projection. We say that this face is \emph{degenerating along this projection}.

\begin{figure}[ht]
    \centering
    \begin{tikzpicture}[x=0.75pt,y=0.75pt,yscale=-1,xscale=1]

\draw   (100,2181.67) -- (170,2181.67) -- (170,2203.67) -- (100,2203.67) -- cycle ;
\draw   (100,2203.67) -- (170,2203.67) -- (170,2243.67) -- (100,2243.67) -- cycle ;
\draw  [dash pattern={on 4.5pt off 4.5pt}] (100,2181.67) -- (170,2181.67) -- (170,2221.67) -- (100,2221.67) -- cycle ;
\draw [color={rgb, 255:red, 208; green, 2; blue, 27 }  ,draw opacity=1 ][line width=1.5]    (100,2181.67) -- (100,2243.67) ;
\draw [color={rgb, 255:red, 208; green, 2; blue, 27 }  ,draw opacity=1 ][line width=1.5]    (170,2181.67) -- (170,2243.67) ;
\draw [color={rgb, 255:red, 208; green, 2; blue, 27 }  ,draw opacity=1 ]   (100,2181.67) -- (100,2203.67) ;
\draw [shift={(100,2203.67)}, rotate = 90] [color={rgb, 255:red, 208; green, 2; blue, 27 }  ,draw opacity=1 ][fill={rgb, 255:red, 208; green, 2; blue, 27 }  ,fill opacity=1 ][line width=0.75]      (0, 0) circle [x radius= 3.35, y radius= 3.35]   ;
\draw [shift={(100,2181.67)}, rotate = 90] [color={rgb, 255:red, 208; green, 2; blue, 27 }  ,draw opacity=1 ][fill={rgb, 255:red, 208; green, 2; blue, 27 }  ,fill opacity=1 ][line width=0.75]      (0, 0) circle [x radius= 3.35, y radius= 3.35]   ;
\draw [color={rgb, 255:red, 208; green, 2; blue, 27 }  ,draw opacity=1 ]   (100,2203.67) -- (100,2243.67) ;
\draw [shift={(100,2243.67)}, rotate = 90] [color={rgb, 255:red, 208; green, 2; blue, 27 }  ,draw opacity=1 ][fill={rgb, 255:red, 208; green, 2; blue, 27 }  ,fill opacity=1 ][line width=0.75]      (0, 0) circle [x radius= 3.35, y radius= 3.35]   ;

\end{tikzpicture}
    \caption{Example of degenerated projection of the cube with three vertices that are sent on aligned points on the boundary}
    \label{degenerated_cube}
\end{figure}
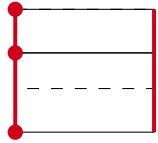

While an extension of the definition using the projection of triplets of vertices would seems natural, it would yield a very complex study, since apriori we would have to study every possible triplet of vertices. Let us simplify this study by showing first that we only need to consider the 2-faces of $P$. Let $W$ be a planar subspace of $\R^d$ and $\pi_W$ the orthogonal projection on $W$, we say that a 2-face $F$ of $P$ is \emph{degenerated for $W$} if the image of $F$ by $\pi_W$ is included in the boundary of $\pi_W(P)$. In particular, if a 2-face is degenerated for $W$, then there exists an edge of $\pi_W(P)$ containing $\pi_W(F)$, and then there are three vertices whose image by $\pi_W$ are aligned on the boundary. The reciprocal is also true: if there are at least three vertices whose image by $\pi_W$ are in the same edge, then there exists a 2-face degenerating for $W$. Indeed, by applying the Lemma 7.10 from \cite{ZIEG} to the edge containing the three aligned images, then the preimage by $\pi$ of this edge intersected with $P$ is a face of $P$ containing at least three vertices, and so it contains at least a 2-face degenerating for $W$. To sum up, like in dimension 3, we need to only pay attention to the 2-faces of $P$.

The downside of this definition of degeneration is that it is difficult to handle. We would rather work with a stronger condition which has a more convenient characterisation. We will forbid every projection sending a 2-face in a segment \emph{anywhere} (not necessarily in the boundary and possibly on a dot). It is clear that condition $(i)$: “no 2-face is sent to a segment” implies condition $(ii)$: “no 2-face is sent in an edge”. The reciprocal is false in higher dimension, where we have the following counter-example: if we project the hypercube $\mathopen[0,1\mathclose]^4$ on the planar space $\vect(e_1+e_2+e_3,2e_3+e_4)$, where $(e_1...,e_d)$ is the canonical basis of $\R^d$, we get a projection looking like the following:

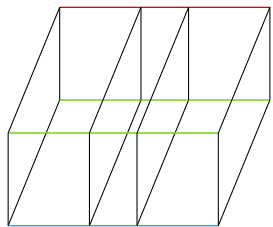
\begin{figure}[ht]
    \centering
    \begin{tikzpicture}[x=0.75pt,y=0.75pt,yscale=-1,xscale=1]
    
\draw    (183,4226.8) -- (157,4290.2) ;
\draw    (159,4226.8) -- (133,4290.2) ;
\draw    (118,4226.8) -- (92,4290.2) ;
\draw [color={rgb, 255:red, 208; green, 2; blue, 27 }  ,draw opacity=1 ]   (118,4180) -- (159,4180) ;
\draw [color={rgb, 255:red, 208; green, 2; blue, 27 }  ,draw opacity=1 ]   (183,4180) -- (224,4180) ;
\draw [color={rgb, 255:red, 208; green, 2; blue, 27 }  ,draw opacity=1 ]   (159,4180) -- (224,4180) ;
\draw [color={rgb, 255:red, 208; green, 2; blue, 27 }  ,draw opacity=1 ]   (118,4180) -- (183,4180) ;
\draw [color={rgb, 255:red, 74; green, 144; blue, 226 }  ,draw opacity=1 ]   (92,4290.2) -- (133,4290.2) ;
\draw [color={rgb, 255:red, 74; green, 144; blue, 226 }  ,draw opacity=1 ]   (157,4290.2) -- (198,4290.2) ;
\draw [color={rgb, 255:red, 74; green, 144; blue, 226 }  ,draw opacity=1 ]   (133,4290.2) -- (198,4290.2) ;
\draw [color={rgb, 255:red, 74; green, 144; blue, 226 }  ,draw opacity=1 ]   (92,4290.2) -- (157,4290.2) ;
\draw [color={rgb, 255:red, 126; green, 211; blue, 33 }  ,draw opacity=1 ]   (118,4226.8) -- (159,4226.8) ;
\draw [color={rgb, 255:red, 126; green, 211; blue, 33 }  ,draw opacity=1 ]   (183,4226.8) -- (224,4226.8) ;
\draw [color={rgb, 255:red, 126; green, 211; blue, 33 }  ,draw opacity=1 ]   (159,4226.8) -- (224,4226.8) ;
\draw [color={rgb, 255:red, 126; green, 211; blue, 33 }  ,draw opacity=1 ]   (118,4226.8) -- (183,4226.8) ;
\draw [color={rgb, 255:red, 126; green, 211; blue, 33 }  ,draw opacity=1 ]   (92,4243.4) -- (133,4243.4) ;
\draw [color={rgb, 255:red, 126; green, 211; blue, 33 }  ,draw opacity=1 ]   (157,4243.4) -- (198,4243.4) ;
\draw [color={rgb, 255:red, 126; green, 211; blue, 33 }  ,draw opacity=1 ]   (133,4243.4) -- (198,4243.4) ;
\draw [color={rgb, 255:red, 126; green, 211; blue, 33 }  ,draw opacity=1 ]   (92,4243.4) -- (157,4243.4) ;
\draw    (92,4243.4) -- (92,4290.2) ;
\draw    (133,4243.4) -- (133,4290.2) ;
\draw    (157,4243.4) -- (157,4290.2) ;
\draw    (198,4243.4) -- (198,4290.2) ;
\draw    (118,4180) -- (92,4243.4) ;
\draw    (118,4180) -- (118,4226.8) ;
\draw    (159,4180) -- (159,4226.8) ;
\draw    (183,4180) -- (183,4226.8) ;
\draw    (224,4180) -- (224,4226.8) ;
\draw    (224,4226.8) -- (198,4290.2) ;
\draw    (159,4180) -- (133,4243.4) ;
\draw    (183,4180) -- (157,4243.4) ;
\draw    (224,4180) -- (198,4243.4) ;
    
    \end{tikzpicture}

    \caption{A degenerated projection of the hypercube}
    \label{hypercube}
\end{figure}

\noindent where the four colored lines are the four 2-faces parallel to $\vect(e_1,e_2)$ that are degenerating. By slightly perturbing the blue and the red 2-face (respectively bottom and top on the drawing), we obtain a new polytope satisfying condition $(ii)$ but not condition $(i)$ for the same projection as shown on Figure \ref{perturbed hypercube} below, where the red 2-face, the blue 2-face and their adjacent 2-faces may be split into two non-coplanar triangles. We do not represent the added edge for the sake of readability of the figure.

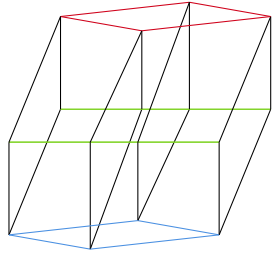
\begin{figure}[ht]
    \centering
    \begin{tikzpicture}[x=0.75pt,y=0.75pt,yscale=-1,xscale=1]
\draw    (369,4226.8) -- (343,4283) ;
\draw    (345,4226.8) -- (319,4297.4) ;
\draw    (304,4226.8) -- (278,4290.2) ;
\draw [color={rgb, 255:red, 208; green, 2; blue, 27 }  ,draw opacity=1 ]   (304,4180) -- (345,4187.2) ;
\draw [color={rgb, 255:red, 208; green, 2; blue, 27 }  ,draw opacity=1 ]   (369,4172.8) -- (410,4180) ;
\draw [color={rgb, 255:red, 208; green, 2; blue, 27 }  ,draw opacity=1 ]   (304,4180) -- (369,4172.8) ;
\draw [color={rgb, 255:red, 74; green, 144; blue, 226 }  ,draw opacity=1 ]   (278,4290.2) -- (319,4297.4) ;
\draw [color={rgb, 255:red, 74; green, 144; blue, 226 }  ,draw opacity=1 ]   (343,4283) -- (384,4290.2) ;
\draw [color={rgb, 255:red, 126; green, 211; blue, 33 }  ,draw opacity=1 ]   (304,4226.8) -- (345,4226.8) ;
\draw [color={rgb, 255:red, 126; green, 211; blue, 33 }  ,draw opacity=1 ]   (369,4226.8) -- (410,4226.8) ;
\draw [color={rgb, 255:red, 126; green, 211; blue, 33 }  ,draw opacity=1 ]   (345,4226.8) -- (410,4226.8) ;
\draw [color={rgb, 255:red, 126; green, 211; blue, 33 }  ,draw opacity=1 ]   (304,4226.8) -- (369,4226.8) ;
\draw [color={rgb, 255:red, 126; green, 211; blue, 33 }  ,draw opacity=1 ]   (278,4243.4) -- (319,4243.4) ;
\draw [color={rgb, 255:red, 126; green, 211; blue, 33 }  ,draw opacity=1 ]   (343,4243.4) -- (384,4243.4) ;
\draw [color={rgb, 255:red, 126; green, 211; blue, 33 }  ,draw opacity=1 ]   (319,4243.4) -- (384,4243.4) ;
\draw [color={rgb, 255:red, 126; green, 211; blue, 33 }  ,draw opacity=1 ]   (278,4243.4) -- (343,4243.4) ;
\draw    (278,4243.4) -- (278,4290.2) ;
\draw    (319,4243.4) -- (319,4297.4) ;
\draw    (343,4243.4) -- (343,4283) ;
\draw    (384,4243.4) -- (384,4290.2) ;
\draw    (304,4180) -- (278,4243.4) ;
\draw    (304,4180) -- (304,4226.8) ;
\draw    (345,4187.2) -- (345,4226.8) ;
\draw    (369,4172.8) -- (369,4226.8) ;
\draw    (410,4180) -- (410,4226.8) ;
\draw    (410,4226.8) -- (384,4290.2) ;
\draw    (345,4187.2) -- (319,4243.4) ;
\draw    (369,4172.8) -- (343,4243.4) ;
\draw    (410,4180) -- (384,4243.4) ;
\draw [color={rgb, 255:red, 208; green, 2; blue, 27 }  ,draw opacity=1 ]   (345,4187.2) -- (410,4180) ;
\draw [color={rgb, 255:red, 74; green, 144; blue, 226 }  ,draw opacity=1 ]   (319,4297.4) -- (384,4290.2) ;
\draw [color={rgb, 255:red, 74; green, 144; blue, 226 }  ,draw opacity=1 ]   (278,4290.2) -- (343,4283) ;

    \end{tikzpicture}
    \caption{A perturbed hypercube providing a counter-example}
    \label{perturbed hypercube}
\end{figure}

\begin{rmk}
In dimension 3 only, conditions $(i)$ and $(ii)$ are equivalent, as this case is the only one where the planes are separating.
\end{rmk}

We will abuse terminology and still say that a 2-face is degenerated for $W$ if it is sent in a line by $\pi_W$. As "being parallel to" is an equivalence relation on the set of 2-faces of $P$, we would rather say "the class of $F$" than "$F$ and its parallel 2-faces if they exist". We say that a class of 2-face is degenerating for $W$ if one of its representative degenerates (and so does every other representatives of that class).

\begin{defin}[Admissibility] \label{admissible}
A planar vector space $W$ is said to be \emph{admissible for $P$} if no class of 2-faces of $P$ is degenerating for $W$. As $W$ is completely determined by its orthogonal, we say that a space $V$ of codimension 2 is  \emph{an admissible orthogonal} if $V^\perp$ is an admissible plane.
\end{defin}

If $F$ is a $k$-face of $P$, we denote $\vect[F]$ the only $k$-dimensional sub-space parallel to $F$. With this definition of admissibility, we have that a planar vector space $W$ is admissible if and only if either $W \cap \vect[F]^\perp$ or $W^\perp \cap \vect[F]$ is reduced to $\{0\}$ (and in this case, both spaces are) for all 2-faces $F$. It is also equivalent to asking a finite number of determinants to be non-zero. This notion will be the one we will use in Definition \ref{new_def}.

\begin{rmk}
    Here, we developed two different notions of degeneration. One could wonder if these two notions would yield different definitions of equiprojectivity. The answer is no, the two possible definitions of equiprojectivity are equivalent, but the proof requires tools from later sections.
\end{rmk}

\begin{rmk}
    It is not the only possible generalisation of equiprojectivity. For instance we may consider projections on sub-spaces of constant dimension at least 3. In this case, we have to determine what is the quantity we want to preserve: it can be either the number of vertices, of edges, of facets or even the combinatorial type. This notion is out of the scope of this article.
\end{rmk}

\section{Some counting}

With Definition \ref{new_def} at hand, let us show that the observation that any 3-dimensional zonotope is equiprojective (see Proposition 2.1 in~\cite{TBLP}) carries over to higher dimensions. Let us fix the dimension $d$. 
\begin{proof}[Proof of observation and bound \eqref{borne}]
Up to translation, a zonotope is a set of the form \[Z = \sum_{g\in \mathcal G} \conv(0,g)\] where $\mathcal G$ is the set of generators of $Z$ and is assumed to be minimal, that is to say no two generators are collinear. Then, it is very easy to compute explicitly the projection of $Z$: \[\pi_W(Z) = \sum_{g\in \mathcal G} \conv(0,\pi_W(g))\] If $W$ is admissible, then it means that the $\pi_W(g)$ are pairwise non-collinear. In this case $\pi_W(Z)$ is a 2-dimensional zonotope with $2|\mathcal G|$ vertices. So $Z$ is $2|\mathcal G|$-equiprojective.

As we can build infinitely many equiprojective polytope with ease (take any prism for instance), we would rather count the number of \emph{combinatorial types} than the number of equiprojective polytopes. Two polytopes have the same combinatorial type if and only if there is a isomorphism between their face-lattice preserving the relation of inclusion.

By applying the Buffière--Pournin bound on the number of combinatorial types of $d$-zonotopes with $\frac{k}{2}$ generators when $k$ is even (see \cite{TBLP}, theorem 1.2), we get that there are at least \[ k^{\frac k2\left(d^2-2d + O_{k\to +\infty}\left( \frac{1}{\log k} \right) \right)} \] different types of $k$-equiprojective polytopes as stated in equation \eqref{borne} in the introduction.
\end{proof}

\section{Walking in the Grassmannian}

The proof of the Hasan--Lubiw characterisation in \cite{MHAL} relies on a lemma which states that one can find a continuous path between two admissible directions of projection such that at most one class of 2-faces degenerates at any time, and there are finitely many degenerations (see Lemma 2 therein). The crucial issue of considering a direction of projection (i.e. a dot on the 2-dimensional sphere) rather than the whole orthogonal of projection, is that we cannot generalise this lemma in higher dimension as our orthogonal is of dimension $d-2$. Before going further, we have to explicit some notion of arcwise connectedness in the Grassmannian.

While there are many realisations of the Grassmannian, we have to take into account the geometry of the set of projections resulting in multiple degeneracies before choosing which realisation we will use. We want to especially avoid these projections as a multiple degeneracy would make tracking the combinatorics of the projection very difficult. In dimension 3, this set is composed only of finitely many points, and we do not need such a machinery. However in higher dimensions, the geometry of this set is much more complex to study. As such, we need to work case by case to build explicit paths in order to avoid such projections. This construction is what limits our choice. For instance, if we work with Plücker coordinates or we represent the Grassmannian as a set of orthogonal projections, it may prove difficult to build a path while satisfying the the different constraints (the Plücker relations and the idempotence and symmetry respectively). Furthermore, we choose not to use the realisation of the Grassmannian as a representable functor or as Schubert cells, since these representations do not provide information about how the vector spaces interact with a given polytope. As such, the most convenient way to do this is to work with free families of vectors, even before quotienting by the relation “the two free families span the same space”. This operation of quotienting will preserve the continuity of our paths if we endow the Grassmannian with its quotient topology. Moreover, using such realisation of the Grassmannian will allow us to precisely handle the polytopes in the next section.

\subsection{The generalisation of a connectedness lemma}

Let us now state the aforementioned definition of connectedness, given here in the general case.

\begin{defin}
    Let $V_0,V_1\subset \R^d$ two vector spaces of dimension $k$. We say that $\gamma : t \in \mathopen[0,1\mathclose] \mapsto \gamma(t) \in (\R^d)^{k} $ is a \emph{continuous arc (or path) between $V_0$ and $V_1$} if: \begin{enumerate}
        \item for all $t \in \mathopen[0,1\mathclose]$, $\gamma(t)$ is a free family
        \item $\vect(\gamma(0))= V_0$ and $\vect(\gamma(1))= V_1$
        \item $\gamma$ is continuous except in a finite number of points $t_1,...,t_n$ such that $0<t_1<...<t_n<1$ and admits limits above and below at these points that satisfy $\vect(\gamma(t_i^-)) = \vect(\gamma(t_i^+)) = \vect(\gamma(t_i))$
    \end{enumerate}
    If such a map exists, we say that \emph{we can continuously go from $V_0$ to $V_1$} and we denote $V_t = \vect(\gamma(t))$.
\end{defin}

\begin{rmk}
    The first condition ensures that $\vect(\gamma(t))$ keeps the same dimension along the path. The third condition allows a finite number of changes of basis.
\end{rmk}

\begin{rmk}
    Like in the classical definition of arcwise connectedness, we can follow a path backward and we can concatenate paths. In other words, if we can go continuously from $V_0$ to $V_1$ and from $V_1$ to $V_2$, then we can go continuously from $V_1$ to $V_0$ and from $V_0$ to $V_2$.
\end{rmk}

One interesting fact with this notion is that the function $t \in \mathopen[0,1\mathclose] \mapsto \pi_{W_t}(v)$ is continuous for all $v$. Indeed, the Gram--Schmidt process is continuous, thus $t \mapsto \pi_{W_t^\perp}(v)$ is a continuous function, and so is $t \mapsto \pi_{W_t}(v)$.

Let us now focus on our case $k = d-2$. Now our goal is to move continuously between any two admissible orthogonals while having at most one degeneration at every time, as stated in Theorem \ref{informal}. Let us rephrase it with the updated terminology: 

\begin{thm}[walking in the Grassmannian]
\label{balade}
    Let $P$ be a polytope of $\R^d$. We can go continuously from any admissible orthogonal to any other admissible orthogonal such that there are only finitely many degenerations along the path and at each of them, only one class of 2-faces is degenerating.
\end{thm}

\begin{rmk}
    Working on the orthogonal of projection is more convenient than working on the projection plane. Indeed, with the orthogonal, we want to avoid spaces of dimension 1 or 2, which we can handle with ease.
\end{rmk}

We divide the proof into lemmas in order to make it resemble an usual proof of pathwise connectedness: we will choose an orthogonal of reference $W_r^\perp$, and then we will go from any admissible orthogonal to $W_r^\perp$. Before we begin the proof of Theorem~\ref{balade}, we need one more geometric notion:

\begin{defin}
    Let $F$ and $F'$ be two non-parallel 2-faces of $P$. Then
    \begin{enumerate}
        \item either $\vect[F]\cap \vect[F'] = \{0\}$, and we say that $F$ and $F'$ are \emph{estranged},
        \item or $\vect[F]\cap \vect[F'] \neq \{0\}$, and this intersection is a vector line spanned by some vector called a \emph{proscribed direction}.
    \end{enumerate}
\end{defin}

\begin{rmk}
    \begin{itemize}
        \item Every 2-face $F$ is such that $\vect[F]$ contains at least one proscribed direction, as edge directions are a special case of proscribed direction. Furthermore, each polytope admits finitely many proscribed directions.
\begin{figure}[ht]
    \centering
    \begin{tikzpicture}[x=0.75pt,y=0.75pt,yscale=-1,xscale=1]

\draw    (359,770.2) -- (404,733.6) ;
\draw    (280,770) -- (359,770.2) ;
\draw    (294,598.2) -- (404,733.6) ;
\draw  [dash pattern={on 4.5pt off 4.5pt}]  (317,714.2) -- (404,733.6) ;
\draw    (294,598.2) -- (359,770.2) ;
\draw  [dash pattern={on 4.5pt off 4.5pt}]  (219,733.8) -- (317,714.2) ;
\draw    (219,733.8) -- (280,770) ;
\draw    (294,598.2) -- (280,770) ;
\draw    (219,733.8) -- (294,598.2) ;
\draw  [dash pattern={on 4.5pt off 4.5pt}]  (294,598.2) -- (317,714.2) ;
\draw  [dash pattern={on 4.5pt off 4.5pt}]  (40,769.8) -- (219,733.8) ;
\draw  [dash pattern={on 4.5pt off 4.5pt}]  (40,769.8) -- (280,770) ;
\draw [color={rgb, 255:red, 74; green, 144; blue, 226 }  ,draw opacity=1 ][fill={rgb, 255:red, 74; green, 144; blue, 226 }  ,fill opacity=1 ]   (322,579.6) -- (15,786.6) ;

\draw (92.29,708.32) node [anchor=north west][inner sep=0.75pt]  [rotate=-325.72] [align=left] {proscribed direction};

\end{tikzpicture}
    \caption{An example of proscribed direction which is not an edge direction}
    \label{fig:dir_proscrite}
\end{figure}
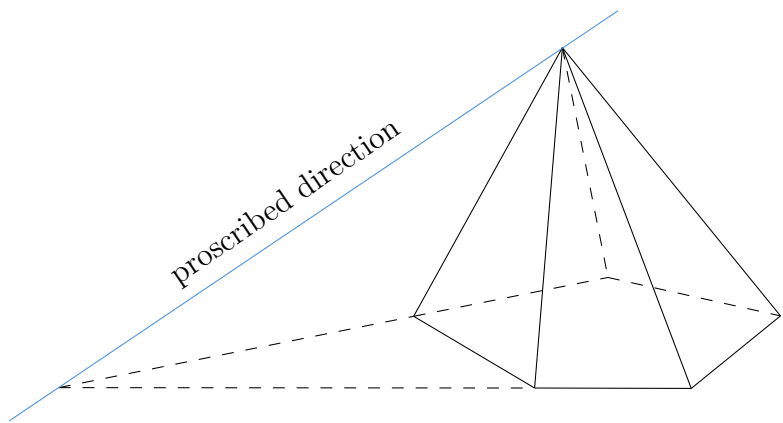
        \item A proscribed direction may not be an edge direction, even in dimension 3 as shown in figure \ref{fig:dir_proscrite}.
        \item These directions are crucial to consider since if $W^\perp$ contains at least one of them, we are sure that several classes are degenerating simultaneously.
        \item In dimension 3, it is enough to avoid all proscribed directions to prove Theorem \ref{balade}, but it is not the case when $d\geq 4$, as we can find example of polytopes with as many pairwise estranged 2-faces degenerating simultaneously as we want (see the next sub-section for an explicit construction).
    \end{itemize}
\end{rmk}

\begin{lemma}[choosing the orthogonal of reference]
\label{step_1}
    By applying an isometry of $\R^d$ to the polytope $P$, we can get the following properties:
    \begin{itemize}
        \item No proscribed direction belongs to $e_1^\perp$. In this case, for each 2-faces $F$, $\vect[F]\cap e_1^\perp$ is a line spanned by some non-zero vector $\eta_F$.
        \item  The last coordinate of every $\eta_F$ is non-zero. We will assume that $\eta_{F,d}=1$
    \end{itemize}
\end{lemma}

\begin{proof}
    When it comes to the proscribed directions, if one of them is unsatisfying, we can apply a rotation modifying only the first coordinate and one of its non-zero coordinate. If the angle of rotation is small enough, this isometry acts like a perturbation of the polytope, and of its proscribed directions. As there are only finitely many such directions, we can choose an angle such that the perturbing isometry will not turn any satisfying proscribed direction into an unsatisfying one. We can use the same process to deal with the last coordinate of the vector $\eta_F$.
\end{proof}

\begin{rmk}
    When the isometry of Lemma \ref{step_1} is applied, there is a one-to-one correspondence between the set of all classes of 2-faces, and the set of all $\eta_F$.
\end{rmk}

Under the conditions of Lemma \ref{step_1}, we can show that $W_r = \vect(e_1,e_d)$ is admissible using the characterization of admissibility in terms of intersections and the fact that the family composed of $\eta_F$ and any proscribed direction of $F$ forms a basis of $\vect[F]$.

In order to go from any admissible orthogonal to $W_r^\perp$, we split the path into two parts. The first one is given by the following lemma:

\begin{lemma}
\label{step_2}
Let $P$ be a polytope satisfying the properties of Lemma \ref{step_1}, then we can go continuously from any admissible orthogonal to another included in $e_1^\perp$ while satisfying the properties of Theorem \ref{balade}.
\end{lemma}

\begin{proof}
    Let $W_0^\perp$ be an admissible orthogonal which is not in $e_1^\perp$. Let $(u_1,...,u_{d-2})$ be a basis of $W_0^\perp$. By Grassmann's formula, we can suppose that $u_2,...,u_{d-2} \in e_1^\perp$, but not $u_1$: we only have to send the latter into $e_1^\perp$. Let $v$ in \[ \mathcal V = \R^d\backslash \left[ e_1^\perp \cup \bigcup_{F~2-face} \vect(W_0^\perp, \eta_F) \right] \] This set is the complementary of a zero-measure set and so is non-empty. Also, it is stable under multiplication by a non-zero scalar, so we can suppose that $u_1-v \in e_1^\perp$ up to rescaling $v$. We consider the path:
    \begin{equation}
    \label{path_1}
        t\in \mathopen[0,1\mathclose] \mapsto (u_1-tv,u_2,...,u_{d-2})
    \end{equation} which is a valid path going continuously from $W_0^\perp$ to some $W_1^\perp\subset e_1^\perp$. We have to check if $W_1^\perp$ is admissible. For any $v\in \mathcal V$, if $W_1^\perp \cap \vect[F] \neq \{0\}$ for some 2-face $F$, then $\mu \eta_F = \lambda_1 (u_1-v) + \lambda_2 u_2 + ... + \lambda_{d-2} u_{d-2}$. As $W_0$ is admissible, we have that $\lambda_1 \neq 0$ and so $v\in \vect(W_0^\perp,\eta_F)$. It is a contradiction, thus $W_1^\perp$ is admissible.
    
    We have to be sure that the other conditions of Theorem \ref{balade} are satisfied. Let $F$ be a 2-face of $P$ and $(f_1,f_2)$ be a basis of $\vect[F]$. We are interested in the number of roots of the function \[ D_F : t\in \mathopen[0,1\mathclose] \mapsto \det(u_1-tv,u_2,..,u_{d-2},f_1,f_2) \] as $D_F(t_0) = 0$ is equivalent to the degeneration of the 2-face $F$ for the projection along $\vect(u_1-t_0v,u_2,...,u_{d-2})$. $D_F$ is an affine function with a non-zero ordinate at the origin by hypothesis on $W_0$. So it admits at most one zero in $\mathopen [0,1  \mathclose]$, whatever the choice of $v$. 
    
    We now need to study what happens when (at least) two different classes are degenerating simultaneously, and how to solve this issue. Let $F$ and $F'$ be two non-parallel 2-faces in that case. We can express their time of degeneration by solving the equations $D_F(t)=0$ and $D_{F'}(t)=0$ for $t$. The solutions coincide when \begin{align*}
        \det(v,u_2,...,u_{d-2},f_1,f_2)&\det(u_1,...,u_{d-2},f'_1,f_2') =\\
        & \det(v,u_2,...,u_{d-2},f'_1,f'_2)\det(u_1,...,u_{d-2},f_1,f_2).
     \end{align*} Both sides are defining a linear form in $v$, we denote $\varphi_{F,F'}(v)$ the left hand-side and $\varphi_{F',F}(v)$ the right hand-side. We want to find a $v\in \mathcal{V}$ such that $\varphi_{F,F'}(v) \neq \varphi_{F',F}(v)$.
     
     Let us find a condition under which such vector $v$ exists. If $\varphi_{F,F'}=\varphi_{F',F}$, then they have the same kernel, i.e. $\vect(u_2,..,u_{d-2},f_1,f_2) = \vect(u_2,..,u_{d-2},f'_1,f'_2)$. Reciprocally, if the two spaces are the same, then $\varphi_{F,F'}$ and $\varphi_{F',F}$ are proportional. Evaluating the relation $\varphi_{F,F'}= \lambda \varphi_{F',F}$ in $u_1$ yields $\lambda = 1$, and thus $\varphi_{F,F'}=\varphi_{F',F}$.

     So, if $\vect(u_2,..,u_{d-2},f_1,f_2) \neq \vect(u_2,..,u_{d-2},f'_1,f'_2)$ then it is possible to perturb $v$ so that the perturbation still lies in $\mathcal{V}$ (as it is an opened set) in order to remove (at least) one simultaneous degeneration while not adding new ones, as the zeros of $D_F$ are continuous in $v$. In dimension 3, under the hypothesis of $F$ and $F'$ being non-parallel, this condition of existence of the perturbation of $v$ is always valid. So let us suppose the dimension $d$ to be strictly greater than 3. If the two spaces are the same, then we perturb $u_2$ as follows. We have in particular that $\det(u_2,..,u_{d-2},f_1,f_2,f'_1) = 0$. Let $w\not \in \vect(u_3,..,u_{d-2},f_1,f_2,f'_1)$, then $\det(u_2+\varepsilon w,..,u_{d-2},f_1,f_2,f'_1) \neq 0$ and, up to translating $w$ by a proscribed direction in $F$, we can suppose that $w\in e_1^\perp$. As the zeros of the functions $D_F$ are continuous in $u_2$, it is possible to find an $\varepsilon>0$ small enough such that we do not add new simultaneous degenerations. Furthermore, since being admissible is an opened condition on any basis of $W_0^\perp$, up to reducing yet again $\varepsilon$, the path \[ \gamma : s\in \mathopen[0,\varepsilon\mathclose]\mapsto (u_1,u_2+sw,u_3,...,u_{d-2}) \] is such that $\vect(\gamma(s))$ is admissible for all $s$. This perturbation of $u_2$ may not solve the simultaneous degeneration of $F$ and $F'$. But in this case, we can guarantee that $\varphi_{F,F'} \neq \varphi_{F',F}$ (beware that the expression of the functions depends on the orthogonal of projection!) This ensure the existence of a perturbation of $v$ solving this simultaneous degeneration. 
     
     As there are only finitely many times of degeneration (at most one per class of 2-faces), we can iterate this process of perturbation as long as there are simultaneous degenerations, and by concatenating all the finitely many needed $\gamma$ and a path of the same kind as equation \eqref{path_1}, we get the wanted path.
\end{proof}

Thanks to this lemma, we can now focus only on the admissible orthogonals that are included in $e_1^\perp$. Under the properties of Lemma \ref{step_1}, it will be more convenient now as we will only deal with hyperplanes of $e_1^\perp$ and the degenerations are caused by the lines spanned by the vectors $\eta_F$. Furthermore, a path in $e_1^\perp$ is enough since $W_r^\perp$ is included in $e_1^\perp$.

\begin{lemma}
\label{step_3}
    Let $P$ be a polytope satisfying the properties of Lemma \ref{step_1}, then we can go continuously from any admissible orthogonal in $e_1^\perp$ to $W_r^\perp$ while satisfying the properties of Theorem \ref{balade}.
\end{lemma}

\begin{proof}
    Let $W_0^\perp\subset e_1^\perp$ be an admissible orthogonal which is not $W_r^\perp$. By Gaussian elimination, it admits a basis of the form \[ \begin{pmatrix}
        1& & & & & \\
         &\ddots& & &(0)& \\
        (0)& &1& & & \\
        x_2&...&x_{i-1}& & & \\
         & & &1& & \\
         &(0)& & &\ddots& \\
         & & & & &1
    \end{pmatrix} \] where each column represent a vector of the basis written in $(e_2,...,e_d)$. By the same argument as the perturbation $\gamma$ in the proof of Lemma \ref{step_2}, we can continuously go from $W_0^\perp$ to some perturbation $\tilde W_0^\perp$ while staying admissible along all the path, where $\tilde W_0^\perp$ has a basis \[ \begin{pmatrix}
        1& & & & & \\
         &\ddots& & &(0)& \\
        (0)& &1& & & \\
        x_2&...&x_{i-1}&\varepsilon& & \\
         & & &1& & \\
         &(0)& & &\ddots& \\
         & & & & &1
    \end{pmatrix} \] By Gaussian elimination, this operation allow the line of $x_j$ to go down one row. Iterating this process allows us to only consider basis looking like \[ \Gamma(x) := \begin{pmatrix}
        1& &(0) \\
         &\ddots& \\
        (0)& &1\\
        x_2&...&x_{d-1}\\
    \end{pmatrix} \]
    Then, we consider the path \[ t \in \mathopen [0,1  \mathclose] \mapsto \begin{pmatrix}
        1& &(0) \\
         &\ddots& \\
        (0)& &1\\
        (1-t)x_2&...&(1-t)x_{d-1}\\
    \end{pmatrix} \left(= \Gamma((1-t)x)\right) \] and we want to study the roots of the finitely many functions \[ D_F : t \in \mathopen [0,1  \mathclose] \mapsto \begin{vmatrix}
        1& &(0) &\eta_{F,2} \\
         &\ddots& & \vdots \\
        (0)& &1 & \eta_{F,d-1} \\
        (1-t)x_2&...&(1-t)x_{d-1} &1\\
    \end{vmatrix} \] After computation, we get that $\displaystyle D_F(t) = (1-t)\sum_{i=2}^{d-1}\eta_{F,i}x_i +1 $. They are affine functions and $D_F(1) = 1$ so they are non-zero. They admit at most one root in $\mathopen [0,1  \mathclose]$, and so there are only finitely many degenerations along this kind of path.
    
    If two non-parallel 2-faces are degenerating simultaneously, then \[ \varphi_F(x) = \sum_{i=2}^{d-1}\eta_{F,i}x_i = \sum_{i=2}^{d-1}\eta_{F',i}x_i = \varphi_{F'}(x) \] Where $\varphi_F$ and $\varphi_{F'}$ are linear forms in $x\in \R^{d-2}$. As before, we perturb $(x_i)_i$ in order to break this equality. We only need to find an $(\tilde x_i)_i$ such that $\varphi_F(\tilde x) \neq \varphi_{F'} (\tilde x)$, and then, with an $\varepsilon>0$ small enough, we are able to go continuously from $\vect(\Gamma(x))$ to $\vect(\Gamma(x+\varepsilon \tilde x))$ while being sure we do not add any new simultaneous degeneration. If there is no such $\tilde x$, then $\varphi_F = \varphi_{F'}$. Evaluating this relation in the canonical basis of $\R^{d-2}$ yields that $\eta_{F,i} = \eta_{F',i}$ for all coordinate. Thus $\eta_F = \eta_{F'}$, which is a contradiction. So there exists an $\tilde x$ such that $\varphi_F(\tilde x) \neq \varphi_{F'} (\tilde x)$.
    
    Then, by concatenating the finitely many needed perturbations and the path $t\mapsto \Gamma((1-t)x)$ for some $x$, we get the path we wanted.
\end{proof}

With these lemmas, we can now prove Theorem \ref{balade}:

\begin{proof}[Proof of Theorem \ref{balade}]
    Denote $f$ the isometry of Lemma \ref{step_1}. If $W_0^\perp$ is admissible for $P$, then $f(W_0^\perp)$ is admissible for $f(P)$. By concatenation of the paths from Lemmas~\ref{step_2} and \ref{step_3}, we get a path $\gamma$ going continuously from $f(W_0^\perp)$ to $W_r^\perp$ satisfying the properties we want. Then $f^{-1}\circ\gamma$ is a continuous path going from $W_0^\perp$ to $f^{-1}(W_r^\perp)$ which is also admissible and does not depend on $W_0^\perp$. Then by an argument of arcwise connectedness, we get the desired result.
\end{proof}

The main consequence of this theorem is that we only need to study what happens when one and only one class of 2-faces is degenerating.

\subsection{A polytope with many degeneracies}

Let us turn our attention to the announced example of a polytope admitting arbitrarily many estranged 2-faces that can degenerate simultaneously. We focus first on the 4-dimensional case, as the higher-dimensional examples can be built by recursion on the dimension.

Let $n$ be an integer greater than 2. In order to build a 4-polytope with at least $n$ estranged 2-faces degenerating simultaneously, let us consider the family of planes \[ F_\theta = \vect \left( f_{1,\theta} = \begin{pmatrix}
    \cos \theta \\ \sin \theta \\ \cos \theta \\ \sin \theta
\end{pmatrix} ;f_{2,\theta} = \begin{pmatrix}
    \cos \theta \\ \sin \theta \\ -\sin \theta \\ \cos \theta
\end{pmatrix} \right) \] where $\theta$ belongs to $\left[ -\frac{\pi}{4},\frac{\pi}{4} \right]$. They form a family of pairwise estranged planes and their orthogonal projection on $W = \vect(e_1,e_2)$ are lines.

Let $\varepsilon>0$, the family $Y_\theta = (f_{1,\theta}, (1-\varepsilon)f_{2,\theta}, (1+\varepsilon)f_{2,\theta})$ describes a triangle contained in $F_\theta$ which image by the projection on $W$ is a segment of length $2\varepsilon$ of the same orientation as the ray of angle $\theta$ with respect to $e_1$ and centered on the point $(\cos \theta, \sin \theta)$ (see figure \ref{Y_theta} below).

\begin{figure}[ht]
    \centering
    \begin{tikzpicture}[x=0.75pt,y=0.75pt,yscale=-1,xscale=1]

\draw [line width=1.5]    (123,1285.8) -- (123,1078.2) ;
\draw [shift={(123,1075.2)}, rotate = 90] [color={rgb, 255:red, 0; green, 0; blue, 0 }  ][line width=1.5]    (14.21,-4.28) .. controls (9.04,-1.82) and (4.3,-0.39) .. (0,0) .. controls (4.3,0.39) and (9.04,1.82) .. (14.21,4.28)   ;
\draw [line width=1.5]    (123,1185.3) -- (230,1185.2) ;
\draw [shift={(233,1185.2)}, rotate = 179.95] [color={rgb, 255:red, 0; green, 0; blue, 0 }  ][line width=1.5]    (14.21,-4.28) .. controls (9.04,-1.82) and (4.3,-0.39) .. (0,0) .. controls (4.3,0.39) and (9.04,1.82) .. (14.21,4.28)   ;
\draw   (123,1275.6) .. controls (169.94,1275.6) and (208,1235.17) .. (208,1185.3) .. controls (208,1135.43) and (169.94,1095) .. (123,1095) -- cycle ;
\draw  [dash pattern={on 4.5pt off 4.5pt}]  (213,1095) -- (123,1185.3) ;
\draw  [dash pattern={on 4.5pt off 4.5pt}]  (123,1185.3) -- (213,1275) ;
\draw [color={rgb, 255:red, 74; green, 144; blue, 226 }  ,draw opacity=1 ][fill={rgb, 255:red, 74; green, 144; blue, 226 }  ,fill opacity=1 ]   (190.5,1161.68) -- (217,1151.73) ;
\draw [shift={(217,1151.73)}, rotate = 339.42] [color={rgb, 255:red, 74; green, 144; blue, 226 }  ,draw opacity=1 ][fill={rgb, 255:red, 74; green, 144; blue, 226 }  ,fill opacity=1 ][line width=0.75]      (0, 0) circle [x radius= 3.35, y radius= 3.35]   ;
\draw [shift={(190.5,1161.68)}, rotate = 339.42] [color={rgb, 255:red, 74; green, 144; blue, 226 }  ,draw opacity=1 ][fill={rgb, 255:red, 74; green, 144; blue, 226 }  ,fill opacity=1 ][line width=0.75]      (0, 0) circle [x radius= 3.35, y radius= 3.35]   ;
\draw [color={rgb, 255:red, 74; green, 144; blue, 226 }  ,draw opacity=1 ]   (190,1208.4) -- (217,1218.4) ;
\draw [shift={(217,1218.4)}, rotate = 20.32] [color={rgb, 255:red, 74; green, 144; blue, 226 }  ,draw opacity=1 ][fill={rgb, 255:red, 74; green, 144; blue, 226 }  ,fill opacity=1 ][line width=0.75]      (0, 0) circle [x radius= 3.35, y radius= 3.35]   ;
\draw [shift={(190,1208.4)}, rotate = 20.32] [color={rgb, 255:red, 74; green, 144; blue, 226 }  ,draw opacity=1 ][fill={rgb, 255:red, 74; green, 144; blue, 226 }  ,fill opacity=1 ][line width=0.75]      (0, 0) circle [x radius= 3.35, y radius= 3.35]   ;
\draw [color={rgb, 255:red, 74; green, 144; blue, 226 }  ,draw opacity=1 ][fill={rgb, 255:red, 74; green, 144; blue, 226 }  ,fill opacity=1 ]   (203.75,1156.71) -- (217,1151.73) ;
\draw [shift={(217,1151.73)}, rotate = 339.42] [color={rgb, 255:red, 74; green, 144; blue, 226 }  ,draw opacity=1 ][fill={rgb, 255:red, 74; green, 144; blue, 226 }  ,fill opacity=1 ][line width=0.75]      (0, 0) circle [x radius= 3.35, y radius= 3.35]   ;
\draw [shift={(203.75,1156.71)}, rotate = 339.42] [color={rgb, 255:red, 74; green, 144; blue, 226 }  ,draw opacity=1 ][fill={rgb, 255:red, 74; green, 144; blue, 226 }  ,fill opacity=1 ][line width=0.75]      (0, 0) circle [x radius= 3.35, y radius= 3.35]   ;
\draw [color={rgb, 255:red, 74; green, 144; blue, 226 }  ,draw opacity=1 ]   (217,1218.4) -- (203.5,1213.4) ;
\draw [shift={(203.5,1213.4)}, rotate = 200.32] [color={rgb, 255:red, 74; green, 144; blue, 226 }  ,draw opacity=1 ][fill={rgb, 255:red, 74; green, 144; blue, 226 }  ,fill opacity=1 ][line width=0.75]      (0, 0) circle [x radius= 3.35, y radius= 3.35]   ;
\draw [shift={(217,1218.4)}, rotate = 200.32] [color={rgb, 255:red, 74; green, 144; blue, 226 }  ,draw opacity=1 ][fill={rgb, 255:red, 74; green, 144; blue, 226 }  ,fill opacity=1 ][line width=0.75]      (0, 0) circle [x radius= 3.35, y radius= 3.35]   ;

\end{tikzpicture}
    \caption{In blue, different examples of projection of $Y_\theta$}
    \label{Y_theta}
\end{figure}
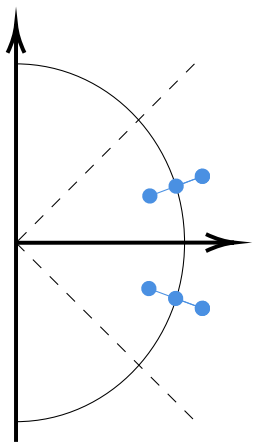

These 2-simplices will be the ones to degenerate for $W$ for some well-chosen values of $\theta$ and $\varepsilon$. Take a regular subdivision of $\left[ -\frac{\pi}{4},\frac{\pi}{4} \right]$ with step $\frac{\pi}{2(n+1)}$. We denote this subdivision $-\frac{\pi}{4} = \alpha_0 < \alpha_1 < ... < \alpha_n < \alpha_{n+1} = \frac{\pi}{4} $. These $\alpha_i$ will be our values of $\theta$. We then translate each $Y_{\alpha_i}$ by some vector so that the image of the 2-simplices are tangent to the unit circle of $\R^2$. And finally, by choosing $\varepsilon$ small enough so that no segment overlap, we obtain some new points $(V_i, V_{i\pm\varepsilon})$. Taking the convex hull of the $3n$ points we have just built makes a 4-polytope $P_n$ with at least $n$ estranged 2-faces that degenerate simultaneously for $W$.

\begin{figure}[ht]
    \centering
    \begin{tikzpicture}[x=0.75pt,y=0.75pt,yscale=-1,xscale=1]

\draw  [fill={rgb, 255:red, 208; green, 2; blue, 27 }  ,fill opacity=0.2 ] (458,1126.4) -- (444,1131.8) -- (420,1146.8) -- (581,1147.8) -- (559,1133.8) -- (544,1127.4) -- (515,1120.4) -- (488,1120.4) -- cycle ;
\draw [line width=1.5]    (359.2,1239.89) -- (651,1239.89) ;
\draw [shift={(654,1239.89)}, rotate = 180] [color={rgb, 255:red, 0; green, 0; blue, 0 }  ][line width=1.5]    (14.21,-4.28) .. controls (9.04,-1.82) and (4.3,-0.39) .. (0,0) .. controls (4.3,0.39) and (9.04,1.82) .. (14.21,4.28)   ;
\draw [line width=1.5]    (499.88,1239.89) -- (500.02,1089) ;
\draw [shift={(500.02,1086)}, rotate = 90.05] [color={rgb, 255:red, 0; green, 0; blue, 0 }  ][line width=1.5]    (14.21,-4.28) .. controls (9.04,-1.82) and (4.3,-0.39) .. (0,0) .. controls (4.3,0.39) and (9.04,1.82) .. (14.21,4.28)   ;
\draw  [dash pattern={on 4.5pt off 4.5pt}] (373.48,1239.89) .. controls (373.48,1239.89) and (373.48,1239.89) .. (373.48,1239.89) .. controls (373.48,1174.21) and (430.07,1120.97) .. (499.88,1120.97) .. controls (569.69,1120.97) and (626.28,1174.21) .. (626.28,1239.89) -- cycle ;
\draw  [dash pattern={on 4.5pt off 4.5pt}]  (626.28,1113.98) -- (499.88,1239.89) ;
\draw  [dash pattern={on 4.5pt off 4.5pt}]  (499.88,1239.89) -- (374.32,1113.98) ;
\draw [color={rgb, 255:red, 208; green, 2; blue, 27 }  ,draw opacity=1 ][line width=1.5]    (420,1146.8) -- (444,1131.8) ;
\draw [color={rgb, 255:red, 208; green, 2; blue, 27 }  ,draw opacity=1 ][line width=1.5]    (559,1133.8) -- (581,1147.8) ;
\draw [color={rgb, 255:red, 208; green, 2; blue, 27 }  ,draw opacity=1 ][line width=1.5]    (458,1126.4) -- (488,1120.4) ;
\draw [color={rgb, 255:red, 208; green, 2; blue, 27 }  ,draw opacity=1 ][line width=1.5]    (515,1120.4) -- (544,1127.4) ;
\draw [line width=1.5]    (488,1120.4) -- (515,1120.4) ;
\draw [line width=1.5]    (420,1146.8) -- (581,1147.8) ;
\draw [line width=1.5]    (444,1131.8) -- (458,1126.4) ;
\draw [line width=1.5]    (544,1127.4) -- (559,1133.8) ;

\end{tikzpicture}
    \caption{Example of the projection of $P_4$ on $W$. In red, the degenerating 2-faces.}
\end{figure}
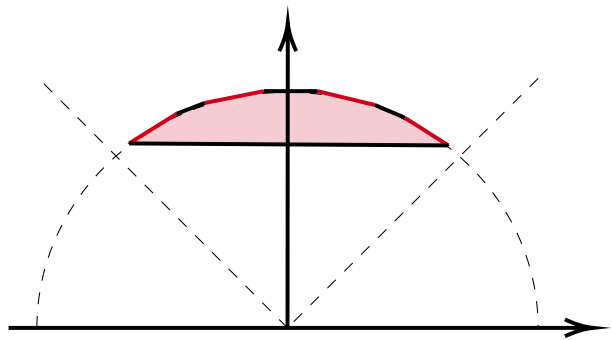

For higher-dimensional examples, we use “hyper-prisms”: let $(e_i)_i$ be the canonical basis of $\R^d$. Let $n$ be fixed and consider the sequence of polytope $(P_{n,d})_{d\geq4}$ defined by \[ \begin{cases}
    P_{n,4} = P_n \\
    P_{n,d+1} = P_{n,d} + \conv(0,e_{d+1} + \rho_{n,d})
\end{cases} \] Where $\rho_{n,d}$ is a small random vector in $\R^{d+1}$ here to help preventing degeneration of faces of dimension greater than 3. For all $n$ and all $d$, $P_{n,d}$ is a $d$-polytope and almost surely it has at least $n$ estranged 2-faces degenerating simultaneously.

\section{Characterisation of equiprojectivity}

First of all, we need a generalisation of the notion of \emph{edge-facet} introduced in \cite{MHAL}. When a 2-face is degenerating and is visible on the boundary of the projection, its disappearing edges can only be compensated by other edges either on the same 2-face, or on another 2-face degenerating on the boundary as well. We say that two parallel 2-faces $F$ and $F'$ are \emph{simultaneously visible} if there exists a projection on a planar space such that only the class of $F$ (and $F'$) is degenerating and $F$ and $F'$ are sent in the boundary of the projection. To take into account the case where $F$ is the only one in its class to degenerate on the boundary, we allow $F'$ to be empty.

Let us note that it is too much to study all possible couples of parallel 2-faces, as we may find some case in dimension $d>3$ where such a couple of 2-faces will never degenerate simultaneously on the boundary. For instance, on the hypercube $\mathopen [0,1 \mathclose]^4$, two parallel 2-faces on the same facet will never be simultaneously visible. If we only represent a degenerated projection of this facet, we can see that we can get the projection of the hypercube by translating the degenerated cube by a vector which is not aligned with any edge of the projected 3-dimensional cube. By studying all the possible cases, we can see that one out of the two 2-faces will be inside the projection (see figure \ref{optimality} below).

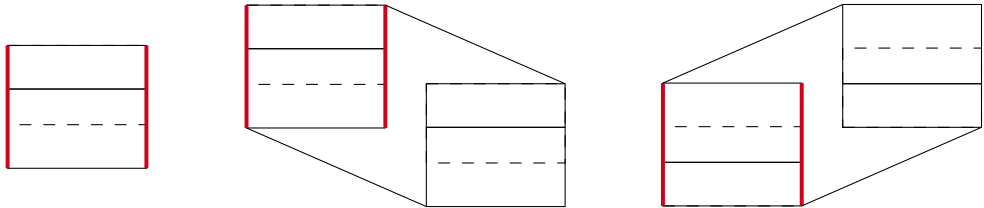
\begin{figure}[ht]
\centering
\begin{tikzpicture}[x=0.75pt,y=0.75pt,yscale=-1,xscale=1]

\draw   (100,2180.67) -- (170,2180.67) -- (170,2202.67) -- (100,2202.67) -- cycle ;
\draw   (100,2202.67) -- (170,2202.67) -- (170,2242.67) -- (100,2242.67) -- cycle ;
\draw  [dash pattern={on 4.5pt off 4.5pt}] (100,2180.67) -- (170,2180.67) -- (170,2220.67) -- (100,2220.67) -- cycle ;
\draw [color={rgb, 255:red, 208; green, 2; blue, 27 }  ,draw opacity=1 ][line width=1.5]    (100,2180.67) -- (100,2242.67) ;
\draw [color={rgb, 255:red, 208; green, 2; blue, 27 }  ,draw opacity=1 ][line width=1.5]    (170,2180.67) -- (170,2242.67) ;
\draw   (220.5,2160.33) -- (290.5,2160.33) -- (290.5,2182.33) -- (220.5,2182.33) -- cycle ;
\draw   (220.5,2182.33) -- (290.5,2182.33) -- (290.5,2222.33) -- (220.5,2222.33) -- cycle ;
\draw  [dash pattern={on 4.5pt off 4.5pt}] (220.5,2160.33) -- (290.5,2160.33) -- (290.5,2200.33) -- (220.5,2200.33) -- cycle ;
\draw [color={rgb, 255:red, 208; green, 2; blue, 27 }  ,draw opacity=1 ][line width=1.5]    (220.5,2160.33) -- (220.5,2222.33) ;
\draw [color={rgb, 255:red, 208; green, 2; blue, 27 }  ,draw opacity=1 ][line width=1.5]    (290.5,2160.33) -- (290.5,2222.33) ;
\draw   (311.17,2200) -- (381.17,2200) -- (381.17,2222) -- (311.17,2222) -- cycle ;
\draw   (311.17,2222) -- (381.17,2222) -- (381.17,2262) -- (311.17,2262) -- cycle ;
\draw  [dash pattern={on 4.5pt off 4.5pt}] (311.17,2200) -- (381.17,2200) -- (381.17,2240) -- (311.17,2240) -- cycle ;
\draw   (430.5,2261.67) -- (500.5,2261.67) -- (500.5,2239.67) -- (430.5,2239.67) -- cycle ;
\draw   (430.5,2239.67) -- (500.5,2239.67) -- (500.5,2199.67) -- (430.5,2199.67) -- cycle ;
\draw  [dash pattern={on 4.5pt off 4.5pt}] (430.5,2261.67) -- (500.5,2261.67) -- (500.5,2221.67) -- (430.5,2221.67) -- cycle ;
\draw [color={rgb, 255:red, 208; green, 2; blue, 27 }  ,draw opacity=1 ][line width=1.5]    (430.5,2261.67) -- (430.5,2199.67) ;
\draw [color={rgb, 255:red, 208; green, 2; blue, 27 }  ,draw opacity=1 ][line width=1.5]    (500.5,2261.67) -- (500.5,2199.67) ;
\draw   (521.17,2222) -- (591.17,2222) -- (591.17,2200) -- (521.17,2200) -- cycle ;
\draw   (521.17,2200) -- (591.17,2200) -- (591.17,2160) -- (521.17,2160) -- cycle ;
\draw  [dash pattern={on 4.5pt off 4.5pt}] (521.17,2222) -- (591.17,2222) -- (591.17,2182) -- (521.17,2182) -- cycle ;
\draw    (220.5,2222.33) -- (311.17,2262) ;
\draw    (290.5,2160.33) -- (381.17,2200) ;
\draw    (500.5,2261.67) -- (591.17,2222) ;
\draw    (430.5,2199.67) -- (521.17,2160) ;

\end{tikzpicture}
\caption{On the left: a degenerated projection of the cube ; On the right and center: two examples of degenerated projection of the hypercube, showing that one or the other red 2-face is sent inside the projection.}
\label{optimality}
\end{figure}

Let us assume that the two facets depicted on the figure \ref{optimality} are the ones parallel to $e_4^\perp$ and the 2-faces in red are parallel to $\vect(e_1,e_2)$. If we have a projection sending both red 2-faces on the boundary, either they are sent in the same edge of the projection, or they are sent in two different parallel edges of the projection. In each case, let us say that their image is parallel to some vector $u$. In the first case (resp. the second case), the image of $e_3$ (resp. $e_4$) is also parallel to $u$, and so whole 3-faces are degenerating, including other 2-faces in different classes.

\begin{defin}[edge-2-faces and their properties]
    Let $P$ be a polytope, let $F$ and $F'$ be two parallel 2-faces that are simultaneously visible (with possibly $F' = \varnothing$) and let $e$ be an edge of $F$. We say that $(e,F,F')$ is an \emph{edge-2-faces}. Furthermore: \begin{itemize}
        \item if $F'\neq \varnothing$, then $P$ intersected with the smallest affine sub-space containing both $F$ and $F'$ is a 3-polytope and $F$ and $F'$ are two of its facets. We define the orientation of every edge-2-faces of the form $(e,F,F')$ or $(e,F',F)$ as in \cite{MHAL}.
        \item if $F' = \varnothing$ we arbitrarily choose the orientation of one edge of $F$, as its edges describe a cyclic path, this allows us to define the orientation of every other edges.
        \item We say that two edge-2-faces $(e,F,F')$ and $(\tilde e,\tilde F,\tilde F')$ \emph{compensate each other} if
        \begin{enumerate}
            \item either $F = \tilde F'$ and $F' = \tilde F$ and $e$ and $\tilde e$ are parallel and have different orientations,
            \item or $F=\tilde F$ and $F' = \tilde F'$ (possibly empty) and $e$ and $\tilde e$ are parallel and have different orientations.
        \end{enumerate}
    \end{itemize}
\end{defin}

\begin{rmk}
    The choice we make in the case where $F'=\varnothing$ does not change the compensation.
\end{rmk}

We divide the proof of Theorem \ref{charact} in two different theorem, both stating one implication of the wanted characterisation rather than only one as in \cite{MHAL}. In our case, we have to check if the geometric intuition used by Masud Hasan and Anna Lubiw carries over to higher dimension. Let us tackle the first implication:

\begin{thm}
\label{frt_implication}
    Let $P$ be a polytope. If its edge-2-faces can be partitioned into compensating pairs, then $P$ is equiprojective.
\end{thm}

\begin{proof}
    Let $W_0$ and $W_1$ be two admissible planes. Our goal is to show that $\pi_{W_0}(P)$ and $\pi_{W_1}(P)$ have the same number of edges. According to Theorem \ref{balade}, we can go continuously from $W_0^\perp$ to $W_1^\perp$ with a finite number of degenerations, and only one class of 2-faces is degenerating each time. It is enough to show that the amount of edges is kept during a single degeneration: rather than working on the whole path, we focus on the local path around a single degeneration. We call these local paths \emph{elementary transformations}.
    
    During such elementary transformation, there are only two different possibilities. Among all the representative of the only class of 2-faces that is degenerating: \begin{itemize}
        \item either none of them are sent in the boundary. Here, it is clear that the size of the projection does not change, and so we can ignore this case,
        \item or one or two of them are sent on the boundary.
    \end{itemize}
    If three or more of them were to be sent on the boundary, then one edge of the projection would contain at least two of the degenerating 2-faces. In that case, the preimage of this edge is a $k$-face of the polytope containing at least two distinct parallel 2-faces, and so $k$ would be greater than 3. Then every other of its 2-faces are degenerating as well, which is excluded. Let's denote $F$ and $F'$ (with possibly $F'=\varnothing$) the two 2-faces of the second case. It means that they are simultaneously visible.
    
    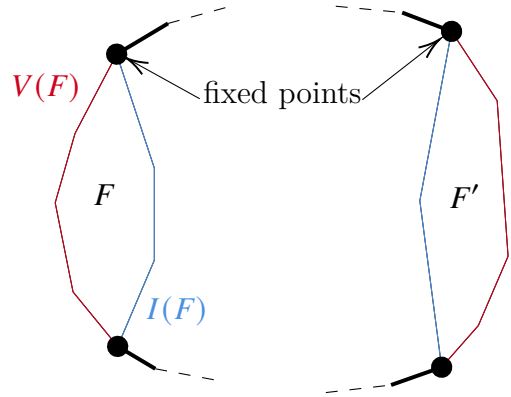
\begin{figure}[ht]
    \centering
    \begin{tikzpicture}[x=0.75pt,y=0.75pt,yscale=-1,xscale=1]

\draw  [color={rgb, 255:red, 0; green, 0; blue, 0 }  ,draw opacity=1 ] (111.5,2486.8) -- (130.5,2544.3) -- (130.5,2590.8) -- (112,2634.3) -- (89.5,2606.8) -- (80.5,2561.8) -- (90.5,2526.8) -- cycle ;
\draw   (280.4,2475.3) -- (303.4,2510.3) -- (308.9,2588.8) -- (293.9,2623.8) -- (275.4,2644.8) -- (264.4,2560.8) -- cycle ;
\draw [color={rgb, 255:red, 208; green, 2; blue, 27 }  ,draw opacity=1 ]   (111.5,2486.8) -- (90.5,2526.8) ;
\draw [color={rgb, 255:red, 208; green, 2; blue, 27 }  ,draw opacity=1 ]   (90.5,2526.8) -- (80.5,2561.8) ;
\draw [color={rgb, 255:red, 208; green, 2; blue, 27 }  ,draw opacity=1 ]   (80.5,2561.8) -- (89.5,2606.8) ;
\draw [color={rgb, 255:red, 208; green, 2; blue, 27 }  ,draw opacity=1 ]   (89.5,2606.8) -- (112,2634.3) ;
\draw [color={rgb, 255:red, 208; green, 2; blue, 27 }  ,draw opacity=1 ]   (280.4,2475.3) -- (303.4,2510.3) ;
\draw [color={rgb, 255:red, 208; green, 2; blue, 27 }  ,draw opacity=1 ]   (303.4,2510.3) -- (308.9,2588.8) ;
\draw [color={rgb, 255:red, 208; green, 2; blue, 27 }  ,draw opacity=1 ]   (308.9,2588.8) -- (293.9,2623.8) ;
\draw [color={rgb, 255:red, 208; green, 2; blue, 27 }  ,draw opacity=1 ]   (293.9,2623.8) -- (275.4,2644.8) ;
\draw [color={rgb, 255:red, 74; green, 144; blue, 226 }  ,draw opacity=1 ]   (111.5,2486.8) -- (130.5,2544.3) ;
\draw [color={rgb, 255:red, 74; green, 144; blue, 226 }  ,draw opacity=1 ]   (130.5,2544.3) -- (130.5,2590.8) ;
\draw [color={rgb, 255:red, 74; green, 144; blue, 226 }  ,draw opacity=1 ]   (130.5,2590.8) -- (112,2634.3) ;
\draw [color={rgb, 255:red, 74; green, 144; blue, 226 }  ,draw opacity=1 ]   (280.4,2475.3) -- (264.4,2560.8) ;
\draw [color={rgb, 255:red, 74; green, 144; blue, 226 }  ,draw opacity=1 ]   (264.4,2560.8) -- (275.4,2644.8) ;
\draw [line width=1.5]    (111.5,2486.8) -- (137.4,2471.3) ;
\draw [shift={(111.5,2486.8)}, rotate = 329.1] [color={rgb, 255:red, 0; green, 0; blue, 0 }  ][fill={rgb, 255:red, 0; green, 0; blue, 0 }  ][line width=1.5]      (0, 0) circle [x radius= 4.36, y radius= 4.36]   ;
\draw [line width=1.5]    (255.5,2465.8) -- (280.4,2475.3) ;
\draw [shift={(280.4,2475.3)}, rotate = 20.88] [color={rgb, 255:red, 0; green, 0; blue, 0 }  ][fill={rgb, 255:red, 0; green, 0; blue, 0 }  ][line width=1.5]      (0, 0) circle [x radius= 4.36, y radius= 4.36]   ;
\draw [line width=1.5]    (249.9,2653.8) -- (275.4,2644.8) ;
\draw [shift={(275.4,2644.8)}, rotate = 340.56] [color={rgb, 255:red, 0; green, 0; blue, 0 }  ][fill={rgb, 255:red, 0; green, 0; blue, 0 }  ][line width=1.5]      (0, 0) circle [x radius= 4.36, y radius= 4.36]   ;
\draw [line width=1.5]    (112,2634.3) -- (131,2645.6) ;
\draw [shift={(112,2634.3)}, rotate = 30.74] [color={rgb, 255:red, 0; green, 0; blue, 0 }  ][fill={rgb, 255:red, 0; green, 0; blue, 0 }  ][line width=1.5]      (0, 0) circle [x radius= 4.36, y radius= 4.36]   ;
\draw    (153.4,2509.3) -- (118.15,2489.77) ;
\draw [shift={(116.4,2488.8)}, rotate = 28.99] [color={rgb, 255:red, 0; green, 0; blue, 0 }  ][line width=0.75]    (10.93,-3.29) .. controls (6.95,-1.4) and (3.31,-0.3) .. (0,0) .. controls (3.31,0.3) and (6.95,1.4) .. (10.93,3.29)   ;
\draw    (235.4,2509.8) -- (272.82,2480.53) ;
\draw [shift={(274.4,2479.3)}, rotate = 141.97] [color={rgb, 255:red, 0; green, 0; blue, 0 }  ][line width=0.75]    (10.93,-3.29) .. controls (6.95,-1.4) and (3.31,-0.3) .. (0,0) .. controls (3.31,0.3) and (6.95,1.4) .. (10.93,3.29)   ;
\draw  [dash pattern={on 4.5pt off 4.5pt}]  (137.4,2471.3) -- (169,2462.6) ;
\draw  [dash pattern={on 4.5pt off 4.5pt}]  (222,2462.6) -- (255.5,2465.8) ;
\draw  [dash pattern={on 4.5pt off 4.5pt}]  (217,2657.6) -- (249.9,2653.8) ;
\draw  [dash pattern={on 4.5pt off 4.5pt}]  (131,2645.6) -- (163,2651.6) ;

\draw (98.5,2551.4) node [anchor=north west][inner sep=0.75pt]   [align=left] {$\displaystyle F$};
\draw (278,2551.4) node [anchor=north west][inner sep=0.75pt]   [align=left] {$\displaystyle F'$};
\draw (154,2498.9) node [anchor=north west][inner sep=0.75pt]   [align=left] {fixed points};
\draw (125,2607.9) node [anchor=north west][inner sep=0.75pt]  [color={rgb, 255:red, 74; green, 144; blue, 226 }  ,opacity=1 ] [align=left] {$\displaystyle I( F)$};
\draw (57.5,2495.4) node [anchor=north west][inner sep=0.75pt]  [color={rgb, 255:red, 208; green, 2; blue, 27 }  ,opacity=1 ] [align=left] {$\displaystyle V( F)$};

\end{tikzpicture}
    \caption{Illustration of the fixed points, visible chains and invisible chains.}
    \label{visibility}
    \end{figure}
    
    On $F$ for instance, there are two special points. As $F$ is degenerating, it is sent exactly on an edge of the projected polygon, and the two extreme points of that edge both correspond to a vertex of $F$. These vertices both belong to a visible edge of $P$ outside of $F$ (in bold lines on figure \ref{visibility}). Since these edges can disappear if and only if there is a degeneration other than that of $F$, they must stay on the boundary through the elementary transformation. In particular, it is also the case for these special vertices of $F$, that we call \emph{fixed points}. The fixed points of $F$ define two chains of edges of $F$. One which was sent on the boundary of the projection and one inside before degeneracy. We call them the \emph{visible chain} $V(F)$ and the \emph{invisible chain} $I(F)$ respectively (see figure \ref{visibility}). We do the same thing with $F'$, with the convention that $V(F')$ and $I(F')$ are both empty if $F'=\varnothing$.
    
    When the degeneracy occurs, there are two different possible scenarios: either the chains $V(F)$ and $V(F')$ stay on the boundary, or they exchange their role with $I(F)$ and $I(F')$ respectively. The first case is not interesting, as nothing changes. We suppose that only the second case happens. In this case, it is important to note that two parallel edges of $F$ can not be in the same set of visibility, as it would imply simultaneous degenerations. In order to conclude, we have to show that \begin{equation}
        \label{eq_1} |V(F)| + |V(F')| = |I(F)| + |I(F')| 
    \end{equation}
    
    Let us study the orientation of the elements of $V(F)$ and $V(F')$. If we look at the 3-polytope which helped define the orientation of the edges of $F$ and $F'$, we notice that the chains $V(F)$ and $V(F')$ have the same orientation (for instance “upward” on figure \ref{visibility}). Hence, every edge in $V(F)$ is compensated either by an edge of $I(F)$ or by an edge of $I(F')$ (and similarly with $F'$). As the edge-2-faces can be partitioned into compensating pairs, equation \eqref{eq_1} holds for each elementary transformation, and so $P$ is equiprojective.
\end{proof}

\begin{rmk}
    The visible and invisible chains depend on the elementary transformation we consider. We will not indicate this dependence unless necessary.
\end{rmk}

To study the reciprocal, we first need a few lemmas that are immediate in dimension 3 in order to handle a polytope. First of all, we need to guarantee that we can build an elementary transformation which swaps its visible and invisible chains.

\begin{lemma}  
    Let $F$ and $F'$ be simultaneously visible 2-faces of $P$. Then there is an elementary transformation which swaps $V(F)$ (resp. $V(F')$) with $I(F)$ (resp. $I(F')$).
\end{lemma}

\begin{proof}
    Let $W^\perp$ be the orthogonal given by the definition of simultaneous visibility of $F$ and $F'$. Let $(u_1,...,u_{d-2})$ be an orthonormal basis of $W^\perp$ such that $u_1\in \vect[F]$. Let $v$ in $(W^\perp + \vect[F])^\perp$, we consider the path: \[ \varphi : t\in \mathopen] -\varepsilon,\varepsilon\mathclose[ \mapsto (u_1+tv,u_2,...,u_{d-2}) \] where $\varepsilon>0$ is small enough such that no other 2-face is degenerating. As \[\det(u_1+tv,u_2,...,u_{d-2},f_1,f_2) = t\det(v,u_2,...,u_{d-2},f_1,f_2)\] where $(f_1,f_2)$ is a basis of $\vect[F]$, we get that all $W_t^\perp = \vect(\varphi(t))$ are admissible for $t\neq 0$, so we have an elementary transformation. When it comes to the swapping, we aim at studying $\pi_{W_t}(u_1)$, as $u_1$ is the cause of the degeneration.
    
    Let $w_1 \in (\vect(u_1,... ,u_{d-2},v))^\perp$, we then have that $w_1\in W_t$ for all $t\in\mathopen] -\varepsilon,\varepsilon\mathclose[$. Let $w_2$ be a vector completing $(w_1)$ into an orthonormal basis of $W_0$. Up to reducing $\varepsilon$ even more, we can assume that $(u_1+tv,u_2,...,u_{d-2},w_1,w_2)$ are all free families. The only thing preventing it to be an orthogonal family is that $w_2$ may not be orthogonal to every previous vectors, more specifically $u_1+tv$. Then we write: \[ \tilde w_2(t) = w_2- t\frac{\langle v,w_2 \rangle}{||u_1+tv||^2}(u_1+tv) \] so $(w_1,\tilde w_2(t))$ is an orthogonal basis of $W_t$. Then we can conclude that: \[ \pi_{W_t}(u_1) = -t\frac{\langle v,w_2 \rangle}{||u_1+tv||^2.||\tilde w_2(t)||^2}\tilde w_2(t) \]
    
    Let $V_1,V_2$ be the two fixed points of this transformation, $(V_1,\overrightarrow{V_1V_2},u_1)$ is a basis of the affine hull of $F$. In it, every vertices in the visible chain have a same-sign coordinate along $u_1$, let us say they are positive, and every vertices in the invisible chain have a negative coordinate along $u_1$. While $\pi_{W_t}(\overrightarrow{V_1V_2})$ does not change “a lot” through the transformation, the evolution of $\pi_{W_t}(u_1)$ is given by the function \[ t\mapsto \left(\langle\pi_{W_t}(u_1),w_1\rangle,\left\langle\pi_{W_t}(u_1),\frac{\tilde w_2(t)}{||\tilde w_2(t)||}\right\rangle\right) = \left(0, -t\frac{\langle v,w_2 \rangle}{||u_1+tv||^2||\tilde w_2(t)||} \right) \] Where the 0 comes from the fact that $\pi_{W^\perp_t}(u_1)$ and thus $\pi_{W_t}(u_1)$ both lie in $\vect(u_1,v)$. Notice the change of sign of the second coordinate at $t=0$. It is this change of sign which guarantees the swapping of $V(F)$ and $I(F)$, and the same study in the affine hull of $F'$ yields the same result.
\end{proof}

\begin{rmk}
    Changing $v$ into $-v$ makes the transformation go backward. As such, we can choose which one of the two chains is the visible one.
\end{rmk}

\begin{corollary}
    The two definitions of degeneracy presented at the end of section 2, that are :
    \begin{enumerate}
        \item “no 2-face is sent on the boundary of the projection”
        \item “no 2-face is sent on a segment anywhere in the projection”
    \end{enumerate}
    yield the same definition of equiprojectivity.
\end{corollary}

\begin{proof}
    It is clear that the condition $(2)$ is stronger than condition $(1)$, so equiprojectivity for condition $(1)$ implies equiprojectivity for condition $(2)$. Given a polytope, we want to study the projections which make 2-faces degenerate only inside the projection.
    
    Let us say that we have $n$ different such degenerating classes. We claim that at least one of them is sent on a segment. Indeed, when a class of 2-faces is sent on a dot, then all of the edges of the representatives of the class are also sent on a dot. This guarantees that every 2-face containing one of these edges are also degenerating. If every degenerating 2-faces are sent on dots, then by proximity, the whole polytope would be sent on a dot, which is absurd.
    
    So let $F$ be a representative of one of the degenerating class sent on a segment. Using the elementary transform developed in the previous proof, we can “un-degenerate” the class of $F$ while avoiding any new degeneracies. In consequence, the combinatorial complexity of the projection does not change and there will be at most $n-1$ classes that are degenerating inside the projection, one of which is sent on a segment. Iterating this process by recursion shows that the combinatorial complexity of a projection satisfying $(2)$ is the same as the one of a projection satisfying $(1)$.
\end{proof}

Now that we can guarantee elementary transformations, the only task left is to choose wisely our chains of visibility.

\begin{lemma}
    \label{chains}
    Let $P$ be a polytope, $F$ and $F'$ (possibly empty) be two of its 2-faces that are simultaneously visible and $e$ be an edge of $F$ such that $F$ has no other edge parallel to $e$. Then there exists two different elementary transformations such that their respective visible chains $V_1(F)$ and $V_2(F)$ satisfy: \[ V_{1}(F) = V_{2}(F) \sqcup \{e\} \text{ or } V_{2}(F) = V_{1}(F) \sqcup \{e\} \]
\end{lemma}

\begin{figure}[ht]
    \centering
    \begin{tikzpicture}[x=0.75pt,y=0.75pt,yscale=-1,xscale=1]

\draw   (412.4,2454.4) -- (452,2518.5) -- (428.4,2646.4) -- (398.9,2659.4) -- (390.4,2561.9) -- cycle ;
\draw   (616.9,2453.4) -- (654.4,2532.4) -- (630.4,2654.9) -- (604.4,2643.9) -- (595.9,2545.9) -- cycle ;
\draw  [dash pattern={on 4.5pt off 4.5pt}]  (398.9,2659.4) -- (438.2,2662) ;
\draw  [dash pattern={on 4.5pt off 4.5pt}]  (630.4,2654.9) -- (669.7,2657.5) ;
\draw  [dash pattern={on 4.5pt off 4.5pt}]  (616.9,2453.4) -- (655.2,2442.9) ;
\draw  [dash pattern={on 4.5pt off 4.5pt}]  (412.4,2454.4) -- (453.2,2447.4) ;
\draw [color={rgb, 255:red, 208; green, 2; blue, 27 }  ,draw opacity=1 ]   (412.4,2454.4) -- (390.4,2561.9) ;
\draw [color={rgb, 255:red, 208; green, 2; blue, 27 }  ,draw opacity=1 ]   (390.4,2561.9) -- (398.9,2659.4) ;
\draw [color={rgb, 255:red, 208; green, 2; blue, 27 }  ,draw opacity=1 ]   (616.9,2453.4) -- (595.9,2545.9) ;
\draw [color={rgb, 255:red, 208; green, 2; blue, 27 }  ,draw opacity=1 ]   (595.9,2545.9) -- (604.4,2643.9) ;
\draw [color={rgb, 255:red, 208; green, 2; blue, 27 }  ,draw opacity=1 ]   (604.4,2643.9) -- (630.4,2654.9) ;
\draw [color={rgb, 255:red, 74; green, 144; blue, 226 }  ,draw opacity=1 ]   (452,2518.5) -- (412.4,2454.4) ;
\draw [color={rgb, 255:red, 74; green, 144; blue, 226 }  ,draw opacity=1 ]   (452,2518.5) -- (428.4,2646.4) ;
\draw [color={rgb, 255:red, 74; green, 144; blue, 226 }  ,draw opacity=1 ]   (428.4,2646.4) -- (398.9,2659.4) ;
\draw [color={rgb, 255:red, 74; green, 144; blue, 226 }  ,draw opacity=1 ]   (654.4,2532.4) -- (630.4,2654.9) ;
\draw [color={rgb, 255:red, 74; green, 144; blue, 226 }  ,draw opacity=1 ]   (616.9,2453.4) -- (654.4,2532.4) ;

\draw (406.3,2631) node [anchor=north west][inner sep=0.75pt]  [color={rgb, 255:red, 74; green, 144; blue, 226 }  ,opacity=1 ] [align=left] {$\displaystyle e$};
\draw (612.3,2629.5) node [anchor=north west][inner sep=0.75pt]  [color={rgb, 255:red, 208; green, 2; blue, 27 }  ,opacity=1 ] [align=left] {$\displaystyle e$};
\draw (412.5,2543.4) node [anchor=north west][inner sep=0.75pt]   [align=left] {$\displaystyle F$};
\draw (616.5,2544.4) node [anchor=north west][inner sep=0.75pt]   [align=left] {$\displaystyle F$};

\end{tikzpicture}
    \caption{Illustration of Lemma \protect\ref{chains}, exaggerated for the sake of readability.}
\end{figure}

\begin{proof}
    Let $W^\perp = \vect(u_1,...,u_{d-2})$, where $u_1\in \vect[F]$ and $(u_1,...,u_{d-2})$ is orthogonal. The fact that an (oriented) edge belong in one of the two chains is characterised by the sign of the inner product of the direction of the edge and a vector orthogonal to the degenerating direction (see figure \ref{inner product}).
    
\begin{figure}[ht]
        \centering
        \begin{tikzpicture}[x=0.75pt,y=0.75pt,yscale=-1,xscale=1]

\draw   (225.7,2798.4) -- (312.7,2806.4) -- (335.2,2878.4) -- (297.7,2911.9) -- (266.7,2922.9) -- (187.2,2910.9) -- (159.2,2865.4) -- (195.7,2806.9) -- cycle ;
\draw    (126.4,2779.1) -- (126.4,2827.16) ;
\draw [shift={(126.4,2829.16)}, rotate = 270] [color={rgb, 255:red, 0; green, 0; blue, 0 }  ][line width=0.75]    (10.93,-3.29) .. controls (6.95,-1.4) and (3.31,-0.3) .. (0,0) .. controls (3.31,0.3) and (6.95,1.4) .. (10.93,3.29)   ;
\draw    (126.4,2779.1) -- (173.6,2779.16) ;
\draw [shift={(175.6,2779.16)}, rotate = 180.07] [color={rgb, 255:red, 0; green, 0; blue, 0 }  ][line width=0.75]    (10.93,-3.29) .. controls (6.95,-1.4) and (3.31,-0.3) .. (0,0) .. controls (3.31,0.3) and (6.95,1.4) .. (10.93,3.29)   ;
\draw    (126,2948.76) -- (376,2949.16) ;
\draw  [dash pattern={on 4.5pt off 4.5pt}]  (159.2,2865.4) -- (158.87,2949) ;
\draw  [dash pattern={on 4.5pt off 4.5pt}]  (187.2,2910.9) -- (187.2,2949.33) ;
\draw  [dash pattern={on 4.5pt off 4.5pt}]  (195.7,2806.9) -- (195.2,2949) ;
\draw  [dash pattern={on 4.5pt off 4.5pt}]  (225.7,2798.4) -- (225.2,2949) ;
\draw  [dash pattern={on 4.5pt off 4.5pt}]  (266.7,2922.9) -- (266.53,2949) ;
\draw  [dash pattern={on 4.5pt off 4.5pt}]  (297.7,2911.9) -- (297.53,2949.33) ;
\draw  [dash pattern={on 4.5pt off 4.5pt}]  (312.7,2806.4) -- (312.2,2949.33) ;
\draw  [dash pattern={on 4.5pt off 4.5pt}]  (335.2,2878.4) -- (334.87,2949.33) ;
\draw [color={rgb, 255:red, 208; green, 2; blue, 27 }  ,draw opacity=1 ]   (159.2,2865.4) -- (194.64,2808.6) ;
\draw [shift={(195.7,2806.9)}, rotate = 121.96] [color={rgb, 255:red, 208; green, 2; blue, 27 }  ,draw opacity=1 ][line width=0.75]    (10.93,-3.29) .. controls (6.95,-1.4) and (3.31,-0.3) .. (0,0) .. controls (3.31,0.3) and (6.95,1.4) .. (10.93,3.29)   ;
\draw [color={rgb, 255:red, 208; green, 2; blue, 27 }  ,draw opacity=1 ]   (195.7,2806.9) -- (223.78,2798.95) ;
\draw [shift={(225.7,2798.4)}, rotate = 164.18] [color={rgb, 255:red, 208; green, 2; blue, 27 }  ,draw opacity=1 ][line width=0.75]    (10.93,-3.29) .. controls (6.95,-1.4) and (3.31,-0.3) .. (0,0) .. controls (3.31,0.3) and (6.95,1.4) .. (10.93,3.29)   ;
\draw [color={rgb, 255:red, 208; green, 2; blue, 27 }  ,draw opacity=1 ]   (225.7,2798.4) -- (310.71,2806.22) ;
\draw [shift={(312.7,2806.4)}, rotate = 185.25] [color={rgb, 255:red, 208; green, 2; blue, 27 }  ,draw opacity=1 ][line width=0.75]    (10.93,-3.29) .. controls (6.95,-1.4) and (3.31,-0.3) .. (0,0) .. controls (3.31,0.3) and (6.95,1.4) .. (10.93,3.29)   ;
\draw [color={rgb, 255:red, 208; green, 2; blue, 27 }  ,draw opacity=1 ]   (312.7,2806.4) -- (334.6,2876.49) ;
\draw [shift={(335.2,2878.4)}, rotate = 252.65] [color={rgb, 255:red, 208; green, 2; blue, 27 }  ,draw opacity=1 ][line width=0.75]    (10.93,-3.29) .. controls (6.95,-1.4) and (3.31,-0.3) .. (0,0) .. controls (3.31,0.3) and (6.95,1.4) .. (10.93,3.29)   ;
\draw [color={rgb, 255:red, 208; green, 2; blue, 27 }  ,draw opacity=1 ]   (158.87,2949) -- (193.2,2949) ;
\draw [shift={(195.2,2949)}, rotate = 180] [color={rgb, 255:red, 208; green, 2; blue, 27 }  ,draw opacity=1 ][line width=0.75]    (10.93,-3.29) .. controls (6.95,-1.4) and (3.31,-0.3) .. (0,0) .. controls (3.31,0.3) and (6.95,1.4) .. (10.93,3.29)   ;
\draw [color={rgb, 255:red, 208; green, 2; blue, 27 }  ,draw opacity=1 ]   (195.2,2949) -- (223.2,2949) ;
\draw [shift={(225.2,2949)}, rotate = 180] [color={rgb, 255:red, 208; green, 2; blue, 27 }  ,draw opacity=1 ][line width=0.75]    (10.93,-3.29) .. controls (6.95,-1.4) and (3.31,-0.3) .. (0,0) .. controls (3.31,0.3) and (6.95,1.4) .. (10.93,3.29)   ;
\draw [color={rgb, 255:red, 208; green, 2; blue, 27 }  ,draw opacity=1 ]   (225.2,2949) -- (310.2,2949.33) ;
\draw [shift={(312.2,2949.33)}, rotate = 180.22] [color={rgb, 255:red, 208; green, 2; blue, 27 }  ,draw opacity=1 ][line width=0.75]    (10.93,-3.29) .. controls (6.95,-1.4) and (3.31,-0.3) .. (0,0) .. controls (3.31,0.3) and (6.95,1.4) .. (10.93,3.29)   ;
\draw [color={rgb, 255:red, 208; green, 2; blue, 27 }  ,draw opacity=1 ]   (312.2,2949.33) -- (332.87,2949.33) ;
\draw [shift={(334.87,2949.33)}, rotate = 180] [color={rgb, 255:red, 208; green, 2; blue, 27 }  ,draw opacity=1 ][line width=0.75]    (10.93,-3.29) .. controls (6.95,-1.4) and (3.31,-0.3) .. (0,0) .. controls (3.31,0.3) and (6.95,1.4) .. (10.93,3.29)   ;
\draw [color={rgb, 255:red, 74; green, 144; blue, 226 }  ,draw opacity=1 ]   (335.2,2878.4) -- (299.19,2910.57) ;
\draw [shift={(297.7,2911.9)}, rotate = 318.22] [color={rgb, 255:red, 74; green, 144; blue, 226 }  ,draw opacity=1 ][line width=0.75]    (10.93,-3.29) .. controls (6.95,-1.4) and (3.31,-0.3) .. (0,0) .. controls (3.31,0.3) and (6.95,1.4) .. (10.93,3.29)   ;
\draw [color={rgb, 255:red, 74; green, 144; blue, 226 }  ,draw opacity=1 ]   (297.7,2911.9) -- (268.58,2922.23) ;
\draw [shift={(266.7,2922.9)}, rotate = 340.46] [color={rgb, 255:red, 74; green, 144; blue, 226 }  ,draw opacity=1 ][line width=0.75]    (10.93,-3.29) .. controls (6.95,-1.4) and (3.31,-0.3) .. (0,0) .. controls (3.31,0.3) and (6.95,1.4) .. (10.93,3.29)   ;
\draw [color={rgb, 255:red, 74; green, 144; blue, 226 }  ,draw opacity=1 ]   (266.7,2922.9) -- (189.18,2911.2) ;
\draw [shift={(187.2,2910.9)}, rotate = 8.58] [color={rgb, 255:red, 74; green, 144; blue, 226 }  ,draw opacity=1 ][line width=0.75]    (10.93,-3.29) .. controls (6.95,-1.4) and (3.31,-0.3) .. (0,0) .. controls (3.31,0.3) and (6.95,1.4) .. (10.93,3.29)   ;
\draw [color={rgb, 255:red, 74; green, 144; blue, 226 }  ,draw opacity=1 ]   (187.2,2910.9) -- (160.25,2867.1) ;
\draw [shift={(159.2,2865.4)}, rotate = 58.39] [color={rgb, 255:red, 74; green, 144; blue, 226 }  ,draw opacity=1 ][line width=0.75]    (10.93,-3.29) .. controls (6.95,-1.4) and (3.31,-0.3) .. (0,0) .. controls (3.31,0.3) and (6.95,1.4) .. (10.93,3.29)   ;
\draw [color={rgb, 255:red, 74; green, 144; blue, 226 }  ,draw opacity=0.5 ]   (334.87,2949.33) -- (299.53,2949.33) ;
\draw [shift={(297.53,2949.33)}, rotate = 360] [color={rgb, 255:red, 74; green, 144; blue, 226 }  ,draw opacity=0.5 ][line width=0.75]    (10.93,-3.29) .. controls (6.95,-1.4) and (3.31,-0.3) .. (0,0) .. controls (3.31,0.3) and (6.95,1.4) .. (10.93,3.29)   ;
\draw [color={rgb, 255:red, 74; green, 144; blue, 226 }  ,draw opacity=0.5 ]   (297.53,2949.33) -- (268.53,2949.02) ;
\draw [shift={(266.53,2949)}, rotate = 0.62] [color={rgb, 255:red, 74; green, 144; blue, 226 }  ,draw opacity=0.5 ][line width=0.75]    (10.93,-3.29) .. controls (6.95,-1.4) and (3.31,-0.3) .. (0,0) .. controls (3.31,0.3) and (6.95,1.4) .. (10.93,3.29)   ;
\draw [color={rgb, 255:red, 74; green, 144; blue, 226 }  ,draw opacity=0.5 ]   (266.53,2949) -- (189.2,2949.32) ;
\draw [shift={(187.2,2949.33)}, rotate = 359.76] [color={rgb, 255:red, 74; green, 144; blue, 226 }  ,draw opacity=0.5 ][line width=0.75]    (10.93,-3.29) .. controls (6.95,-1.4) and (3.31,-0.3) .. (0,0) .. controls (3.31,0.3) and (6.95,1.4) .. (10.93,3.29)   ;
\draw [color={rgb, 255:red, 74; green, 144; blue, 226 }  ,draw opacity=0.5 ]   (187.2,2949.33) -- (160.87,2949.02) ;
\draw [shift={(158.87,2949)}, rotate = 0.67] [color={rgb, 255:red, 74; green, 144; blue, 226 }  ,draw opacity=0.5 ][line width=0.75]    (10.93,-3.29) .. controls (6.95,-1.4) and (3.31,-0.3) .. (0,0) .. controls (3.31,0.3) and (6.95,1.4) .. (10.93,3.29)   ;
\draw   (126.4,2779.1) -- (134.07,2779.1) -- (134.07,2786.77) -- (126.4,2786.77) -- cycle ;

\draw (83.2,2796.8) node [anchor=north west][inner sep=0.75pt]   [align=left] {$\displaystyle u_{1}( t)$};
\draw (151,2779.13) node [anchor=north west][inner sep=0.75pt]   [align=left] {$\displaystyle \tilde{v}( t)$};

\end{tikzpicture}
        \caption{Illustration of the characterisation of the chains}
        \label{inner product}
    \end{figure}

    With that in mind, we aim at finding two orthogonals such that the only vector whose sign of the inner product described above changes is the orientation of $e$. Let us denote it $\vec e$ and suppose it is a unit vector. Let $v \in \vect[F]$ a unit vector orthogonal to $\vec e$. Then we can suppose that $u_1 = \vec e + \lambda v$ with $\lambda>0$, up to changing $v$ into $-v$. We consider the transformation \[ t\in \mathopen[0,\lambda+\varepsilon \mathclose] \mapsto (u_1-tv,u_2,...,u_{d-2})=(u_1(t),u_2,...,u_{d-2}) \] Along this path, $F$ stays degenerated while every other 2-face of $P$ degenerates at most once. Thus, we can choose $\varepsilon>0$ small enough such that there is only one time of degeneration that occurs on $\mathopen[ \lambda-\varepsilon,\lambda+\varepsilon \mathclose]$ (and it occurs at $\lambda$) other than the degeneration of $F$, as it occurs on the whole segment.
    
    The family $(u_1(t),v)$ is free for all $t$. Let us turn it into an orthogonal one $(u_1(t),\tilde v(t))$ : \[ \tilde v(t) = -\frac{\lambda-t}{||u_1(t)||^2}\vec e + \left(1+ \frac{(\lambda-t)^2}{||u_1(t)||^2} \right)v \] Now, let $\vec e_1 = \vec e,\vec e_2,...,\vec e_N$ be all the directions of the edges of $F$. We are interested in the sign changes of the functions $t\mapsto \langle \vec e_i,\tilde v(t) \rangle $ since a change of sign means a change of visibility chain for $e_i$. With our choice of $\varepsilon$, we know that there is only one such change that could occurs on $\mathopen[ \lambda-\varepsilon,\lambda+\varepsilon \mathclose]$ because of $\vec e$. And, in this case we have \[ \langle \vec e,\tilde v(t) \rangle = -\frac{(\lambda-t)}{||u_1(t)||^2} \]
    
    So between $\lambda-\varepsilon$ and $\lambda+\varepsilon$, the only difference in the chains comes from $e$. According to the previous lemma and the remark, we can find two different elementary transformations such that their visible chains satisfy $V_{1}(F) = V_{2}(F) \sqcup \{e\}$.
\end{proof}

We discuss later in the proof of the reciprocal of Theorem \ref{frt_implication} what happen if a parallel 2-face $F'$ is involved:

\begin{thm}
    Let $P$ be a polytope. If $P$ is equiprojective, then its edge-2-faces can be partitioned into pairs that compensate each other.
\end{thm}

\begin{proof}
    It is exactly what is done in \cite{MHAL}, that we will repeat with our notations for the sake of completeness. An edge-2-faces $(e,F,F')$ lives in a group with at most three other “compatible” edges-2-faces: $(\tilde e,F,F')$ with $\tilde e \neq e $, $(e',F',F)$ and $(\tilde e',F',F)$ with $\tilde e' \neq e' $ (if they exist). Then our goal is to partition such group. By contraposition, suppose that there is a group containing the edge-2-faces $(e,F,F')$ that can not be partitioned. There are only four possible cases: \begin{enumerate}
        \item no other edge of $F$ or $F'$ are parallel to $e$ ($F'$ is possibly empty).
        \item $F$ has no other edge parallel to $e$, and $F'$ has exactly one with the same orientation.
        \item $F$ has no other edge parallel to $e$ and $F'$ has two.
        \item $F$ has another edge parallel to $e$ and $F'$ has only one
    \end{enumerate}
    By symmetry of the roles of $F$ and $F'$ in the last two cases, it is enough to suppose that $F$ has no other edge parallel to $e$. With the previous lemma, we can build two different elementary transformations such that $V_{1}(F) = V_{2}(F) \sqcup \{e\}$. We show that for one or the other elementary transformation, we have that \begin{equation} \label{eq_2}
        |V_i(F)| + |V_i(F')| \neq |I_i(F)| + |I_i(F')|
    \end{equation} or equivalently, in order to have all instances of $F$ on the same side: \[ |V_i(F)| - |I_i(F)| \neq |I_i(F')| - |V_i(F')| \]
    
    By contradiction, suppose that we have equality in both cases. Subtracting them yields \[ |V_1(F)| - |V_2(F)| + |I_2(F)| - |I_1(F)| = |V_2(F')| - |V_1(F')| + |I_1(F')| - |I_2(F')| \] In every case, the left hand-side is equal to 2. When it comes to the right hand-side:
    
    In the first case, the chains do not change between the two transformations, and so the right hand-side is equal to 0.
    
    In the second case, the edge parallel to $e$ was in $V_1(F')$ but ended up in $I_2(F')$, and so the right hand-side is equal to -2.
    
    In the third case, the two edges of $F'$ parallel to $e$ lie in different chains of visibility. After the change of visibility of $e$, they exchange theirs. So the left hand-side is equal to 0.
    
    Thus, in every case, we get a contradiction. So one or the other equality in equation \eqref{eq_2} is not satisfied, and $P$ is not equiprojective.
\end{proof}

Hence the characterisation of Theorem \ref{charact}. Similarly to the 3-dimensional Hasan--Lubiw characterisation, this one could lead to an algorithm that can check if a given polytope is equiprojective. The main issue in the high-dimensional case would be listing all the couples of simultaneously visible 2-faces we need to consider. Furthermore, as the work of Théophile Buffière and Lionel Pournin in \cite{TBLP} relies on the result of \cite{MHAL}, this generalisation gives us hope to generalise their work as well.

\bigskip

\textbf{Acknowledgement.} The author of this article is grateful to Lionel Pournin, for his careful reading and the pieces of advice he gave, and to Andrea Sportiello, for a very interesting discussion regarding the Grassmannian.

\printbibliography

@book{ZIEG,
    author = {Günter M. Ziegler},
    title = {Lectures on polytopes},
    publisher = {Springer} ,
    year = {1995},
    edition = {7} }

@article{MHAL,
    author = {Masud Hasan and Anna Lubiw},
    title = {Equiprojective polyhedra} ,
    journal = {Computational Geometry} ,
    year = {2008},
    pages = {148--155},
    volume = {40}
}

@article{TBLP,
    author = {Théophile Buffière and Lionel Pournin} ,
    title = {Many equiprojective polytopes} ,
    journal = {Discrete and computational geometry} ,
    year = {2025},
    volume = {74},
    pages = {337--357}
}

@article{SHEPH,
    author = {Geoffrey C. Shephard} ,
    title = {Twenty problems on convex polyhedra part I} ,
    journal = {The Mathematical Gazette} ,
    year = {1968},
    volume = {52},
    number = {380},
    pages = {136--147}
}

@article{SHEPH2,
    author = {Geoffrey C. Shephard} ,
    title = {Twenty problems on convex polyhedra part II} ,
    journal = {The Mathematical Gazette} ,
    year = {1968},
    volume = {52},
    number = {382},
    pages = {359--367}
}

@unpublished{Black,
  title={Random Shadows of Fixed Polytopes},
  author={Alexander E. Black and Francisco Criado},
  year={2024},
    archivePrefix = {arXiv},
    eprint = {2406.06936},
    primaryClass = {math.CO}
}

@article{DDSH,
    author = {Daniel Dadush and Sophie Huiberts},
    title = {A friendly smooth analysis of the simplex method},
    journal = {Proceedings of the 50th Annual ACM SIGACT Symposium on Theory of Computing} ,
    year = {2018}, 
    pages = {390--403} }

@article{TR,
    author = {Thomas Rothvoss},
    title = {The Matching Polytope has Exponential Extension Complexity},
    journal = {Journal of the ACM},
    volume = {64},
    number = {6},
    year = {2017},
    pages = {41:1--41:19}
    }

@article{Exp_Low,
    author = {Samuel Fiorini and Serge Massar and Sebastian Pokutta and Hans Raj Tiwary and Ronald de Wolf},
    title = {Exponential Lower Bounds for Polytopes in Combinatorial Optimization},
    journal = {Journal of the ACM},
    volume = {62},
    number = {2},
    year = {2015},
    pages = {17:1--17:23}
    }

@book{Pb_in_geom,
    author = {Hallard T. Croft and Kenneth J. Falconer and Richard K. Guy},
    title = {Unsolved problems in geometry},
    year = {1991},
    publisher = {Springer New York, NY}
    }

@article{MH,
    author ={Masud Hasan and Mohammad Monoar Hossain and Alejandro López-Ortiz and Sabrina Nusrat and Saad Altaful Quader and Nabila Rahman},
    title = {Some New Equiprojective Polyhedra},
    year = {2022},
    journal = {Geocombinatorics},
    volume = {31},
    number ={4},
    pages = {170--177}
}

@article{moser,
    author = {Jeffrey C. Lagarias and Yusheng Luo and Arnau Padrol} ,
    title = {Moser's Shadow problem} ,
    journal = {L'Enseignement Mathématique} ,
    year = {2018},
    pages = {477--496},
    volume = {64},
    number = {2}
}

@article{lin_op_1,
    author = {Jean Cardinal and Raphael Steiner},
    title = {Inapproximability of shortest paths on perfect matching polytopes},
    journal = {Mathematical Programming},
	year = {2025},
	volume = {210},
	number = {1},
	pages = {147--163}
}

@article{lin_op_2,
    author = {Alexander E. Black and Jes\'us A. De Loera and Sean Kafer and Laura Sanità},
    title = {On the Simplex Method for 0/1-Polytopes},
    journal = {Mathematics of Operations Research},
    year = {2025},
    volume = {50},
    number = {2},
    pages = {1398--1420}
}

@article{lin_op_3,
    author = {Michele Barbato and Roland Grappe and Mathieu Lacroix and Emiliano Lancini},
    title = {Box-total dual integrality and edge-connectivity},
    journal = {Mathematical Programming},
    year = {2023},
    volume = {198},
    number = {1},
    pages = {307--336}
}

@article{lin_op_4,
    author = {Antoine Deza and Shmuel Onn and Sebastian Pokutta and Lionel Pournin},
    title = {Kissing Polytopes},
    journal = {SIAM Journal on Discrete Mathematics},
    year = {2024},
    volume = {38},
    number = {4},
    pages = {2643-2664}
}

@article{alg,
    author = {Jean-Louis Loday},
    title = {Realization of the Stasheff polytope},
    journal = {Archiv der Mathematik},
    year = {2004},
    volume = {83},
    number = {3},
    pages = {267--278}
}

@book{conv_an,
    author = {Matthias Beck and Sinai Robins},
    title = {Computing the continuous discretely},
    year = {2015},
    publisher = {Springer}
}

@article{DDNH,
    author = {Daniel Dadush and Nicolai Hähnle},
    title = {On the shadow simplex method for curved polyhedra},
    journal = {Discrete \& Computational Geometry},
    year = {2016},
    volume = {56},
    number = {4},
    pages = {882--909}
}

@article{lien_1,
    author = {Antoine Deza and Lionel Pournin and Noriyoshi Sukegawa},
    title = {The diameter of lattice zonotopes},
    journal = {Proceedings of the American Mathematical Society},
    year = {2020},
    volume = {148},
    number = {8},
    pages = {3507--3516}
}

@article{lien_2,
    author = {Francisco Santos},
    title = {A counter example to the Hirsh conjecture},
    journal = {Annals of Mathematics},
    year = {2012},
    volume = {176},
    number = {1},
    pages = {383--412}
}

@article{lien_3,
    author = {Lei Xue},
    title = {A Proof of Grünbaum’s Lower Bound Conjecture for general polytopes},
    journal = {Israel Journal of Mathematics},
    year = {2021},
    volume = {245},
    pages = {991--1000}
}

@article{affine_1,
    author = {Katalin Berlow and Marie-Charlotte Brandenburg and Chiara Meroni and Isabelle Shankar },
    title = {Intersection bodies of polytopes},
    journal = {Contribution to Algebra and Geometry},
    year = {2022},
    volume = {63},
    pages = {419--439}
}

@article{affine_2,
    author = {Marie-Charlotte Brandenburg and Jesús A. De Loera and Chiara Meroni},
    title = {The best way to slice a Polytope},
    journal = {Mathematics of Computation},
    year = {2025},
    volume = {94},
    pages = {1003--1042}
}

@unpublished{affine_3,
    author = {Marie-Charlotte Brandenburg and Chiara Meroni},
    title = {Combinatorics of slices of cubes},
    year = {2025},
    archivePrefix = {arXiv},
    eprint = {2510.09265},
    primaryClass = {math.CO}
}

@article{affine_4,
    author = {Lionel Pournin},
    title = {Shallow sections of the hypercube},
    journal = {Israel Journal of Mathematics},
    year = {2023},
    volume = {255},
    pages = {685--704}
}

\end{document}